\documentclass[10pt]{article}
\usepackage[utf8]{inputenc}
\usepackage[usenames,dvipsnames,svgnames,table]{xcolor}	
\usepackage[]{amsmath}
\usepackage{cases}
\usepackage{mathrsfs}
\usepackage{empheq}
\usepackage[colorlinks=true]{hyperref}
\hypersetup{colorlinks,
citecolor=brown, linkcolor=blue, linktocpage
}
\usepackage{amssymb}
\usepackage{bbm}
\usepackage{enumerate}   
\usepackage{dsfont}
\usepackage{geometry}
\usepackage{amsfonts}
\usepackage{color}
\usepackage[]{amsthm} 
 \usepackage[shortlabels]{enumitem}
 \usepackage{cleveref}

\def\beq{\begin{eqnarray}}
\def\eeq{\end{eqnarray}}
\def\be*{\begin{eqnarray*}}
\def\ee*{\end{eqnarray*}}

\newtheorem{Theorem}{Theorem}[section]
\newtheorem{Lemma}[Theorem]{Lemma}
\newtheorem{Proposition}[Theorem]{Proposition}

\newtheorem{Definition}[Theorem]{Definition}
\newtheorem{Remark}[Theorem]{Remark}

\newtheorem{Assumption}[Theorem]{Assumption}

\newcommand{\Pp}{\mathcal P}

\newcommand{\rmi}{{\rm (i)$\>\>$}}
\newcommand{\rmii}{{\rm (ii)$\>\>$}}
\newcommand{\rmiii}{{\rm (iii)$\>\>$}}
\newcommand{\rmiv}{{\rm (iv)$\>\>$}}

\def \E{\mathbb{E}}
\def \F{\mathbb{F}}
\def \G{\mathbb{G}}

\def \N{\mathbb{N}}

\def \P{\mathbb{P}}
\def \Q{\mathbb{Q}}
\def \R{\mathbb{R}}

\def \W{\mathbb{W}}

\def\Ec{{\cal E}}

\def\Gc{{\cal G}}

\def\Lc{{\cal L}}

\def\Pc{{\cal P}}

\def\Wc{{\cal W}}

\def \d {\delta}
\def \eps {\varepsilon}

\def \no {\noindent}

\def \d{{\rm d}}
\def \1{\mathds{1}}

\def \tp {\tilde{p}}
\def \tq {\tilde{q}}

\numberwithin{equation}{section}

\usepackage{xspace}

\title{Mean-field limit of  particle systems with absorption}      

\author{Gaoyue Guo\footnote{Université Paris-Saclay CentraleSupélec, Laboratoire MICS and CNRS FR-3487. gaoyue.guo@centralesupelec.fr}, \ Maxime Latypov\footnote{Université Paris-Saclay CentraleSupélec, Laboratoire MICS and CNRS FR-3487.  maxime.latypov@centralesupelec.fr} \   and \  
Milica Toma\v sevi\'c\footnote{CNRS, CMAP  Ecole polytechnique, Institut Polytechnique de Paris, 91120 Palaiseau, France. milica.tomasevic@polytechnique.edu, corresponding author}
}

\date{}                      
\begin{document}

\maketitle

\begin{abstract}
Singular particle systems whose singularity may originate from singular coefficients (intrinsic nature of particles) or hitting times (cross-interactions between particles), have  emerged as a highlighted research topic and attracted abundant attention. Following the work of \cite{ORT2020} in which singularity is due to forces of strong attraction or repulsion between particles, this paper puts focus on the other type of singularity, where one-dimensional particles interact in mean-field way through a bounded kernel and are removed from the system once hitting  some barrier. Such singularity due to hitting times of the given barrier has been well studied in a steam of papers under different settings, e.g. \cite{DIRT1, DIRT2, HL2017, CF2018,HLS2019, NS2019, CRS2023, HPRRS2023}, etc, and our novelty is two fold: (1) The mean-field limit has a strong solution and the corresponding nonlinear Fokker-Planck equation admits a classical solution which in turns yields the regularity of the survival probability with respect to time; (2) We adopt a fixed-point argument that is based on PDE analysis and applied to sub-probability distributions to prove uniqueness of the solution to the mean-field limit. In addition, we show that the particle system admits a weak solution and establish the propagation of chaos by proving the tightness through \textit{Partial Girsanov transforms}.

\vspace{3mm}

\no {\bf Keywords: }~~~ interacting particle system, mean-field limit, Fokker-Planck equation, propagation of chaos
\end{abstract}

\section{Introduction}\label{sec:intro}

Singular behaviour in large interacting particle systems arises when the collective dynamics exhibit breakdowns such as blow-up, loss of regularity or macroscopic mass loss, and is a central topic in mathematical physics and probability. One fruitful way to analyse such systems is through mean-field limits: the high-dimensional random dynamics of $N$ particles are approximated, as $N\to\infty$, by a lower-dimensional nonlinear (McKean–Vlasov) model in which a typical particle interacts only with the evolving law of the system. 

In many applications, and in particular in models of \emph{systemic risk}, the relevant singularities are driven by hitting times of a boundary. In this setting, an individual particle represents, for instance, the capital of a financial institution, which is removed from the system when it hits a default barrier, and the cumulative loss of mass feeds back into the dynamics of the survivors. Such models have been studied in various forms in \cite{DIRT1,DIRT2,NS2019,HLS2019,CRS2023,BGTZ2020,BGTZ2022,DGT2023}, where a major focus is on how contagion through boundary losses can lead to cascades and phase transitions in the large population limit. Hambly and Ledger~\cite{HL2017} analyse a system of absorbing diffusions on the half-line whose diffusion coefficient depends on the current loss process and show convergence of the empirical measure to the solution of a nonlinear stochastic heat equation with Dirichlet boundary condition, together with pathwise uniqueness for the associated McKean–Vlasov problem. Campi and Fischer~\cite{CF2018} introduce $N$-player games and mean-field games with absorption on first exit from a bounded domain, where players interact through a renormalised empirical measure of the survivors; they prove existence of mean-field equilibria and show that these generate approximate Nash equilibria for the underlying games. More recently, Hambly, Petronil, Reisinger, Rigger and S{\o}jmark~\cite{HPRRS2023} study contagious McKean–Vlasov equations with common noise, in which feedback from the loss process is smoothed in time through a kernel and a singular regime with jumps is obtained as the kernel collapses to a Dirac mass, together with convergence of the smoothed systems to relaxed (and in some cases strong) solutions of the singular limit.

A different form of singularity comes from strong attractive or repulsive forces between particles, where blow-up is caused by high concentrations rather than by absorption at the boundary. Olivera, Reis and Tugaut~\cite{ORT2020} study such McKean–Vlasov dynamics with singular drift coefficients motivated by biological aggregation, and show that the corresponding particle systems or mean-field limits may cease to exist globally in time when particles become too concentrated. Our work is instead concerned with singularity created by absorbing boundaries: we consider one-dimensional particles which interact through a bounded mean-field drift and are removed from the system as soon as they hit the origin. In this sense, our paper complements \cite{ORT2020} by focusing on the ``hitting-time'' rather than the ``drift'' type of singularity.

More precisely, we study a system of $N$ real-valued processes $X^{N,1},\dots,X^{N,N}$ on $(0,\infty)$ whose dynamics are given by
\begin{equation}\label{eq:p-s}
X^{N,i}_t = Z^{N,i} + \int_0^t \iota(X^{N,i}_s)\,\frac{1}{N}\sum_{j=1}^N b\bigl(s,X^{N,i}_s,X^{N,j}_s\bigr)\,\iota(X^{N,j}_s)\,\mathrm{d}s
+\int_0^t \sqrt{2}\,\iota(X^{N,i}_s)\,\mathrm{d}W^i_s,\qquad t\ge0,\ i=1,\dots,N,
\end{equation}
where $\iota(x):=\mathbf{1}_{\{x>0\}}$, $b:\R_+\times\R^2\to\R$ is a bounded measurable drift, $(Z^{N,1},\dots,Z^{N,N})$ are exchangeable initial positions in $\R_+$ and $W^1,\dots,W^N$ are independent Brownian motions. Once a particle hits the absorbing barrier at $0$ it stays there forever and ceases to interact, so that $X^{N,i}_t=X^{N,i}_{\tau_i^N\wedge t}$ for all $t\ge0$, where $\tau_i^N:=\inf\{t\ge0:\,X^{N,i}_t\le0\}$ is the first hitting time of zero. Formally, as $N\to\infty$ the empirical measure of the particles is expected to converge to the law of the solution of the McKean–Vlasov SDE
\begin{equation}\label{eq:lim}
\mathrm{d}X_t = \iota(X_t)\Bigl(\int_{(0,\infty)} b(t,X_t,y)\,\mu_t(\mathrm{d}y)\Bigr)\mathrm{d}t
+ \iota(X_t)\sqrt{2}\,\mathrm{d}W_t,\qquad \mu_t:=\mathcal{L}(X_t).
\end{equation}

Our first main contribution is to develop a complete well-posedness and propagation of chaos theory for this model under natural conditions on $b$. At the particle level we show that, for bounded measurable $b$, the system \eqref{eq:p-s} admits a unique strong solution and that the sequence of empirical measures is tight. At the mean-field level we prove that any limit point solves the McKean–Vlasov SDE \eqref{eq:lim} and, under additional regularity of $b$, that this equation admits a unique weak solution; this yields propagation of chaos. Under a Lipschitz condition on $b$ we further establish a quantitative \emph{strong} propagation of chaos in $\mathcal{W}_2$ with optimal rate $C/N$ uniformly on compact time intervals; see Theorems~\ref{thm:strong-system}–\ref{thm:wasserstein-chaos}. Compared with \cite{HL2017,HPRRS2023}, which also analyse absorbing systems related to systemic risk, our coefficients depend on the full spatial configuration through a bounded kernel $b(t,x,y)$ rather than only on the loss process, and we obtain a classical (rather than merely distribution-valued) description of the limit via a nonlinear Fokker–Planck equation together with a direct PDE-based uniqueness argument.

The second main contribution is methodological. We show that the law $\mu_t$ of any solution to \eqref{eq:lim} decomposes as
\[
\mu_t(\mathrm{d}x) = \Bigl(1-\int_0^\infty p(t,y)\,\mathrm{d}y\Bigr)\delta_0(\mathrm{d}x) + p(t,x)\,\mathrm{d}x,
\]
where the sub-probability density $p$ solves a nonlinear Fokker–Planck equation with homogeneous Dirichlet boundary condition at $0$. We construct a classical solution $p$ and set up a fixed-point argument on a suitable space of sub-probability densities which yields uniqueness for this PDE and, in turn, for the McKean–Vlasov SDE \eqref{eq:lim}. This PDE viewpoint also provides information on the regularity of the survival probability $t\mapsto\mathbb{P}(X_t>0)$ and suggests natural extensions to higher-dimensional domains and multiple absorbing barriers.

On the probabilistic side, the identification of the mean-field limit and the proof of quantitative propagation of chaos rely on a refined Girsanov strategy. While a full Girsanov transform from the interacting system to independent killed Brownian motions would not be uniform in $N$, we introduce \emph{partial Girsanov transforms} that modify only finitely many coordinates at a time and keep the Radon–Nikodym derivatives under control as $N\to\infty$. This allows us to exploit the fact that individual particles behave like Brownian motions stopped at zero under an appropriate change of measure, and plays a key role both in the tightness analysis and in the martingale characterisation of limit points.

Although our focus is on uncontrolled systems, our results are also relevant for the mean-field games literature with absorption. Campi and Fischer~\cite{CF2018} treat controlled diffusions with absorption through a renormalised empirical measure and conditional laws, under a nondegeneracy assumption on the diffusion coefficient; by contrast, we work with uncontrolled but possibly degenerate dynamics and concentrate on the well-posedness and quantitative propagation of chaos for the underlying McKean–Vlasov model, which can serve as a building block for future controlled versions of systems with hitting-time interactions.

In general, interacting particle systems are studied across diverse applications, including the dynamics of granular media \cite{BCCP1998, BGG2013}, mathematical biology \cite{BCM2007, FJ}, economics and social networks \cite{CDL2013, HMC2006}, and deep neural networks \cite{MMN2018, RVE2018}. A more detailed exposition can be found in \cite{CD11998, CD21998, MFR1, MFR2}. We also note other models of interacting systems with hitting times, such as integrate-and-fire models in neurosciences \cite{DIRT1, DIRT2}, and related financial models \cite{HLS2019}.

\paragraph{Plan of the paper.}
The rest of the paper is structured as follows. In Section~\ref{sec:main} we state our assumptions on the drift $b$, present the main results for the particle system and for the McKean–Vlasov limit, and outline the proof strategy. Section~\ref{sec:thm1} is devoted to the existence and convergence of the particle system \eqref{eq:p-s}, including the construction of partial Girsanov transforms and the proof of tightness. In Section~\ref{sec:thm2} we establish the well-posedness of the limiting McKean–Vlasov SDE \eqref{eq:lim} via the nonlinear Fokker–Planck equation and the parametrix-based fixed-point argument, and we derive the strong propagation of chaos estimate.



\paragraph{Notation.}
For a generic Polish space $E$, denote by $C(E)$ (resp. $\Pc(E)$) the set of continuous functions (resp. probability measures) on $E$. For every $q\ge 1$, define $\Pc_q(E)\subset \Pc(E)$ to be the subset of elements admitting finite $q^{\rm th}$ moment, and $\Wc_q$ to be the corresponding \emph{Wasserstein distance} of order $q$ on $\Pc_q(E)$. In what follows, we define the  empirical measures 
\beq \label{def:empirical}
\mu^N:=\frac1N\sum_{i=1}^N \delta_{X^{N,i}} &\mbox{and}& \mu^N_t:=\frac1N\sum_{i=1}^N \delta_{X^{N,i}_t},\quad \forall t\ge 0, 
\eeq 
that take values respectively in  $\Pp(\Omega)$ and $\Pp(\R)$, where $\Omega:=C(\mathbb{R}_+)$.

Throughout the article, we denote by $C^2(\mathbb{R};[0,1])$ the space of twice continuously differentiable functions from $\mathbb{R}$ to $[0,1]$. 

\section{Main results}\label{sec:main}

Before presenting the main results, we introduce the necessary assumptions regarding the drift term $b$.

\begin{Assumption} \label{ass:1}
The drift coefficient $b:\R_+\times\R^2\to\R$ is assumed to satisfy the following conditions:
\begin{itemize}
    \item[\rmi] \textbf{Boundedness:} $b$ is a measurable and bounded function.
    \item[\rmii] \textbf{Regularity near zero:} For each fixed $t \ge 0$, the function $(x_1, x_2) \mapsto b(t, x_1, x_2)$ is continuous almost everywhere on $\R^2$. Furthermore, for every fixed $x \ge 0$ and $t \ge 0$, the limits $\lim_{y\to 0^+} b(t, x, y)$ and $\lim_{y\to 0^+} b(t, y, x)$ exist.
    \item[\rmiii] \textbf{Hölder continuity:} Both $b$ and its partial derivative with respect to the first spatial variable, $\partial_{x_1} b$, are bounded and Hölder continuous in their spatial arguments $(x_1, x_2)$, uniformly in the time variable $t$.
    \item[\rmiv] \textbf{Lipschitz continuity:} The function $b$ is Lipschitz continuous in its spatial variables, uniformly in time. That is, there exists a constant $L>0$ such that, for all $t\ge0$ and all $(x_1, x_2), (x'_1, x'_2) \in\R^2$,
    $$
    |b(t,x_1, x_2)-b(t,x'_1, x'_2)|\le L\big(|x_1-x'_1|+|x_2-x'_2|\big).
    $$
\end{itemize}
\end{Assumption}

We now define the notions of strong and weak solutions to the nonlinear stochastic differential equation \eqref{eq:lim}. Let $(\Omega, \mathcal{F}, (\mathcal{F}_t)_{t\ge 0}, \mathbb{P})$ be a filtered probability space supporting a standard Brownian motion $W = (W_t)_{t \ge 0}$, and let $F = (F_t)_{t \ge 0}$ denote the canonical process $F_t(\omega) = \omega(t)$ for $\omega \in \Omega = C([0, \infty), \R)$.

\begin{Definition}[Strong and Weak Solutions]\label{def:MP}
\begin{itemize}
    \item[\rmi] A stochastic process $X = (X_t)_{t \ge 0}$ is called a \textbf{strong solution} to \eqref{eq:lim} if it is adapted to the filtration generated by the Brownian motion $W$, and the equality \eqref{eq:lim} holds almost surely for all $t \ge 0$.
    \item[\rmii] A probability measure $\Q \in \mathcal{P}(\Omega)$ is termed a \textbf{weak solution} to \eqref{eq:lim} with initial law $\rho \in \mathcal{P}(\R_+)$ if the canonical process $F = (F_t)_{t \ge 0}$ under $\Q$ satisfies the following conditions:
    \begin{enumerate}
        \item[(a)] \textbf{Initial condition:} The law of $F_0$ under $\Q$, denoted $\Q_0 := \Q \circ F_0^{-1}$, coincides with the initial measure $\rho$.
        \item[(b)] \textbf{Martingale property:} For every test function $\varphi \in C^2_c(\R)$ (i.e., twice continuously differentiable with compact support), the process $M^\varphi = (M^\varphi_t)_{t \ge 0}$ defined by
        $$
        M^\varphi_t := \varphi(F_t)-\varphi(F_0)-\int_0^t \iota(F_s)\left( \varphi'(F_s) \left( \int_{\R_+} b(s,F_s,y)\Q_s (\d y) \right)+ \varphi''(F_s)\right)\d s,\quad \forall t\ge 0,
        $$
        where $\Q_s := \Q \circ F_s^{-1}$ is the time-$s$ marginal law of $F$, is a martingale under $\Q$ with respect to the filtration $\F = (\mathcal{F}_t)_{t \ge 0}$.
        \item[(c)] \textbf{State space constraint:} For all $t\ge 0$, $\Q(F_t \in \R_+) = 1$.
    \end{enumerate}
\end{itemize}
\end{Definition}

Our first result concerns the well-posedness of the interacting particle system \eqref{eq:p-s}.

\begin{Theorem}[Well-posedness of the Particle System \eqref{eq:p-s}] \label{thm:strong-system}
Under Assumption~\ref{ass:1}-(i), if the initial positions $Z^{N,1},\ldots, Z^{N,N}$ are square-integrable random variables, then the particle system \eqref{eq:p-s} admits a unique strong solution $(X^{N,1},\ldots, X^{N,N})$. In particular, for each particle $i \in \{1, \ldots, N\}$, $X_t^{N,i}=X^{N,i}_{t\wedge\tau_i}$ holds for all $t\ge 0$, where $\tau_i=\inf\{t\ge0:X_t^{N,i}\le0\}$ is the first time the $i$-th particle hits zero.
\end{Theorem}

Theorem \ref{thm:strong-system} guarantees the existence and uniqueness of a strong solution for the $N$-particle system \eqref{eq:p-s}, which also implies the existence of a unique weak solution. We now investigate the behavior of the system's empirical measure, $\mu^N$ defined in \eqref{def:empirical}, as $N \to \infty$. The following theorem establishes the convergence of the empirical measure to the solution of the mean-field equation \eqref{eq:lim}, demonstrating the propagation of chaos phenomenon.

\begin{Theorem}[Convergence and Propagation of Chaos for \eqref{eq:p-s}] \label{thm:p-s}
\begin{enumerate}
    \item[1)] Let Assumption \ref{ass:1}-(i) hold. Then, for each $N\ge 1$, there exists a unique weak solution $(X^{N,1},\ldots, X^{N,N})$ to the particle system \eqref{eq:p-s}.
    \item[2)] Additionally, suppose Assumption~\ref{ass:1}-(ii) holds and the initial empirical measure $\mu^N_0=\frac1N\sum_{i=1}^N \delta_{Z^{N,i}}$ converges in probability as $N \to \infty$ to a deterministic measure $\rho\in \mathcal{P}(\R_+)$. Then, any limit point $\mu$ of the sequence of empirical measures $(\mu^N)_{N\geq 1}$ (in the space of probability measures on $C([0, \infty), \R)$) is almost surely a weak solution to the mean-field equation \eqref{eq:lim} with initial law $\rho$.
    \item[3)] If, in addition, Assumption~\ref{ass:1}-(iii) holds and the mean-field equation \eqref{eq:lim} admits a unique weak solution $X$ (see Theorem \ref{thm:lim}), then propagation of chaos holds: the empirical measure $\mu^N$ converges in probability to the law of $X$.
\end{enumerate}
\end{Theorem}

Next, we address the well-posedness of the limiting mean-field SDE \eqref{eq:lim}, which necessitates stronger conditions on the drift $b$.

\begin{Theorem}[Well-posedness of the Mean-Field Limit \eqref{eq:lim}] \label{thm:lim}
Let Assumption \ref{ass:1}-(iii) hold. Suppose the initial law $\mathcal{L}(Z)\in \mathcal{P}_2(\R_+)$ admits a density $\rho:\R_+\to\R_+$ satisfying the condition
$$ \int_0^1\frac{\rho(x)}{x^2}\d x < \infty. $$
Then the mean-field SDE \eqref{eq:lim} has a unique weak solution $X$.
\end{Theorem}

Finally, we provide a quantitative result on the convergence of the empirical measure towards the solution of the mean-field equation, often referred to as strong propagation of chaos.

\begin{Theorem}[Strong Propagation of Chaos (Quantitative Rate)]\label{thm:wasserstein-chaos}
Assume Assumptions~\ref{ass:1}-(i), (iv) hold. Let $\mu^N_t$ be the empirical measure of the particle system \eqref{eq:p-s} at time $t$, and let $\mu_t$ be the law of the unique solution $X_t$ to the mean-field equation \eqref{eq:lim} at time $t$. Then, for every $T>0$, there exists a constant $C>0$ such that
\begin{equation}\label{eq:uniform-sup-W2}
\mathbb{E}\Big[\sup_{t\in[0,T]}\mathcal{W}_2(\mu_t^N,\mu_t)^2\Big]\le \frac{C}{N},
\end{equation}
where $\mathcal{W}_2$ denotes the 2-Wasserstein distance.
\end{Theorem}

\paragraph{Strategy of the proof.}

For Theorem~\ref{thm:lim}, the proof of the existence for the McKean-Vlasov SDE  relies on Schauder's fixed-point theorem. One delicate point here  is that the measure
that solves it accumulates dynamically mass in zero, hence the functional on the space of flows of probability measures for which we would naturally exhibit a fixed point contains an indicator function. This is not continuous for Wasserstein distance for probability measures, so the space on which we exhibit a fixed point is enlarged to take into account this dynamical accumulation of mass in zero.  It becomes the product space of flows of probability measures and $\Omega_T$.\\

 As for the uniqueness, we study the  marginal distributions of any solution $X$ to \eqref{eq:lim}. For $t>0$, we first show that the law of $X_t$ is of the form $p(t,x)dx + (1-\int p(t,x)dx ) \delta_0(dx)$, where $p(t,\cdot)$ is a sub-probability density satisfying a parametric Fokker-Planck equation whose coefficients depend on $p$ itself. We show the one-to-one relation between $X$ and such parametric equation, and the uniqueness of $X$ is equivalent to that of $p$ satisfying this PDE. Due to the non-linearity, we consider PDEs of parametric coefficients and estimate the explicit dependency of their solutions on the 
parameters. By introducing a suitably defined operator, we prove that the operator is a contraction and thus the uniqueness result follows. \\

 For Theorem \ref{thm:p-s}, the existence follows through an application of a Girsanov-Cameron-Martin transform. The boundedness of the coefficient $b$ easily gives the tightness of the sequence of the empirical measures. The identification of the limit is done through the martingale approach and is more delicate due to the singularity the stopping time introduces in the drift. We will thus regularize  $\iota(\cdot)$ and show that regularized and non-regularized problems are close. In the latter, it will be crucial to use the fact that individual particles behave like stopped Brownian motions modulo a change of measure. To make such a change uniform in $N$, we adopt the strategy of \cite{MT2020} and introduce the \textit{Partial Girsanov transforms} (see the discussion after Lemma~\ref{Tightness:crit}). 
 
We finish this part with a following remark. We believe that the first two parts of  Theorem~\ref{thm:p-s} could be adapted to the so called $L^q-L^p$  assumption on $b$ with more involved computations as in \cite{MT2020}. We leave possible generalizations of assumptions on $b$ for a future work.

\section{Proof of Theorem \ref{thm:strong-system}}\label{sec:proof1}
Section \ref{sec:proof1} is reserved for the proof of Theorem \ref{thm:strong-system}. We first  construct an explicit solution solving \ref{thm:strong-system} for the existence result, and then we show that this solution is indeed unique. 

    \begin{proof}
    By Assumption~\ref{ass:1}-(i), the drift is bounded and measurable, hence the reduced (alive–set) SDEs satisfy Assumptions~\ref{ass:segment-solve-diff} and \ref{ass:segment-uniq-diff} (see remark \ref{rem:segment-uniq}). Existence then follows from Proposition~\ref{prop:gluing}, and pathwise uniqueness from Proposition~\ref{prop:global-uniq}. 
    \qed
    \end{proof}
 
We begin by isolating the only properties we actually need on the drift \(b\) in order to build the solution by survival–interval gluing.

\begin{Assumption}[Segment–wise strong solvability]\label{ass:segment-solve-diff}
For every nonempty $A\subset\{1,\dots,N\}$, initial time $s\in[0,T]$, initial vector $(x^i)_{i\in A}\in\R_+^{A}$, and Brownian motions $(W^i)_{i\in A}$ on a filtered probability space, the reduced SDE
\[
\begin{cases}
\d X_t^i=\dfrac1N\sum_{j\in A} b\bigl(t,X_t^i,X_t^j\bigr)\,\d t+\sqrt{2}\,\d W_t^i,\quad t\in[s,T],\\
X_s^i=x^i,\quad i\in A,
\end{cases}
\]
admits at least one strong solution adapted to the given filtration.
\end{Assumption}

\begin{Assumption}[Segment–wise pathwise uniqueness]\label{ass:segment-uniq-diff}
Under the same set-up as above, if $(X_t^i)_{i\in A}$ and $(\tilde X_t^i)_{i\in A}$ are two strong solutions of
\[
\begin{cases}
\d X_t^i=\dfrac1N\sum_{j\in A} b\bigl(t,X_t^i,X_t^j\bigr)\,\d t+\sqrt{2}\,\d W_t^i,\quad t\in[s,T],\\
X_s^i=x^i,\quad i\in A,
\end{cases}
\]
driven by the same Brownian motions and with the same initial vector, then $X_t^i=\tilde X_t^i$ a.s. for all $t\in[s,T]$, $i\in A$. 
\end{Assumption}

\begin{Remark}[On Assumptions \ref{ass:segment-solve-diff} and \ref{ass:segment-uniq-diff}] 
Assumption \ref{ass:segment-solve-diff} (segment-wise strong solvability) and \ref{ass:segment-uniq-diff} (segment-wise pathwise uniqueness) hold under any of the following standard hypotheses:
\begin{itemize}
    \item If $b$ is bounded and measurable in all its variables $(t,x,y)$, then on each alive set $A$ the reduced drift vector is bounded and measurable, and the diffusion is nondegenerate and constant. By the Krylov--Zvonkin theory, the reduced SDE admits a unique strong solution adapted to the given filtration (see, e.g., \cite{IW1989,Figalli2008,RZ2021}).
    \item More generally, if $b$ satisfies an $L_t^q$–$L_x^p$ integrability condition of Krylov--R\"ockner type, namely $b\in L^q([0,T];L^p(\R))$ with $\frac{d}{p}+\frac{2}{q}<1$ (here $d=1$), then segment-wise strong solvability also holds (see, e.g., \cite{RZ2021,MT2020}). In particular $b\in L^q_tL^\infty_x$ with $q>2$ yields segment-wise strong well-posedness.
\end{itemize}
Assumption \ref{ass:segment-solve-diff} and \ref{ass:segment-uniq-diff} also hold in higher dimension (specifically for $d>1$) when $b$ is globally Lipschitz in its space variables, uniformly in time, which yields classical pathwise uniqueness for the reduced SDE on each alive set (see, e.g., \cite{KSBook}).
\label{rem:segment-uniq}
\end{Remark}

\begin{Proposition}[Global Absorbed Existence]\label{prop:gluing}
Under Assumption \ref{ass:segment-solve-diff}, the full absorbed particle system
\[
\begin{cases}
X_t^{N,i}
=Z^{N,i}
+\displaystyle\int_{0}^t
 \frac{1}{N}\sum_{j=1}^N b\bigl(r,X_r^{N,i},X_r^{N,j}\bigr)\,
 \1_{\{\tau^{N,j}>r\}}
 \,\d r
+\sqrt2\,W_t^i,\\
\tau^{N,i}=\inf\{t\ge0: X_t^{N,i}\le0\},
\end{cases}
\]
admits at least one strong solution \(\{X^{N,i}\}_{i=1}^N\) on \([0,T]\).
\end{Proposition}

    \begin{proof}
    Fix $T>0$. By Assumption~\ref{ass:segment-solve-diff}, the non-absorbed $N$-particle SDE (full index set $A=\{1,\dots,N\}$, initial time $0$) admits a strong solution on $[0,T]$; denote it by $(Y^{1,1},\dots,Y^{1,N})$. For $i=1,\dots,N$, set $\sigma_i:=\inf\{t\ge0:\,Y^{1,i}_t\le 0\}$, $\sigma_{(0)}:=0$, and $\sigma_{(1)}:=\min_{1\le i\le N}\sigma_i$. Almost surely the exit times are pairwise distinct, hence $I_{(1)}:=\{i:\sigma_i=\sigma_{(1)}\}$ is a singleton and $A_1:=\{1,\dots,N\}\setminus I_{(1)}$.

    On $[\sigma_{(0)},\sigma_{(1)})$, all coordinates are strictly positive, so the absorbed dynamics and the non-absorbed one coincide. If $\sigma_{(1)}<\infty$, condition on $\mathcal{F}_{\sigma_{(1)}}$ and consider the reduced system on the alive index set $A_1$ with initial data $Y^{1,i}_{\sigma_{(1)}}$ and shifted Brownian motions. By Assumption~\ref{ass:segment-solve-diff} (now with $A=A_1$ and initial time $s=\sigma_{(1)}$), there exists a strong solution $(Y^{2,i})_{i\in A_1}$ on $[\sigma_{(1)},T]$. Define
    \[
    \sigma_{(2)}:=\inf\{t>\sigma_{(1)}:\ \exists\,i\in A_1\ \text{with }Y^{2,i}_t=0\},\quad
    I_{(2)}:=\{i\in A_1:\ \sigma_i=\sigma_{(2)}\},\quad
    A_2:=A_1\setminus I_{(2)}.
    \]

    Proceeding inductively, suppose $\sigma_{(0)}<\cdots<\sigma_{(k-1)}$ and $A_{k-1}$ are defined. On $[\sigma_{(k-1)} ,T]$, apply Assumption~\ref{ass:segment-solve-diff} to the reduced system with alive set $A_{k-1}$, initial time $s=\sigma_{(k-1)}$, and initial condition given by the left limits at $\sigma_{(k-1)}$, to obtain a strong solution $(Y^{k,i})_{i\in A_{k-1}}$. Set
    \[
    \sigma_{(k)}:=\inf\{t>\sigma_{(k-1)}:\ \exists\,i\in A_{k-1}\ \text{with }Y^{k,i}_t=0\},\quad
    I_{(k)}:=\{i\in A_{k-1}:\ \sigma_i=\sigma_{(k)}\},\quad
    A_k:=A_{k-1}\setminus I_{(k)}.
    \]
    Since at each step at least one index is removed and there are only $N$ indices, the construction stops after at most $N$ steps.

    Define the global trajectories by concatenation:
    \[
    Y^i_t :=
    \begin{cases}
    Y^{1,i}_t, & t\in[0,\sigma_{(1)}),\\
    Y^{2,i}_t, & t\in[\sigma_{(1)},\sigma_{(2)}),\\
    \ \vdots & \\
    0, & t\ge \sigma_i.
    \end{cases}
    \]
    By construction, $Y$ is adapted and continuous on each segment and at the junctions; for $i\in I_{(m)}$, $\lim_{t\uparrow \sigma_{(m)}}Y^i_t=0$, and for $i\in A_m$, $\lim_{t\uparrow \sigma_{(m)}}Y^i_t=Y^{m,i}_{\sigma_{(m)}}>0$. Moreover, on each interval $[\sigma_{(m-1)},\sigma_{(m)})$ the system solves the reduced non-absorbed SDE with $\iota\equiv 1$ on the alive set, hence, globally, the absorbed particle SDE holds in the integral sense on $[0,T]$. Extending by $0$ after the last absorption yields a strong solution on $[0,T]$. Since $T>0$ was arbitrary, we obtain a global strong solution on $[0,\infty)$.
    \qed
    \end{proof}

    \begin{Proposition}[Global pathwise uniqueness]\label{prop:global-uniq}
    Under Assumptions \ref{ass:segment-solve-diff} and \ref{ass:segment-uniq-diff}, the absorbed particle system admits at most one strong solution driven by a given Brownian family and initial condition. 
    \end{Proposition}

    \begin{proof}
    Let $X=(X^{1},\dots,X^{N})$ and $Y=(Y^{1},\dots,Y^{N})$ be two strong solutions with the same initial data and Brownian motions. Set the first common exit time
    \[
    \sigma_{(1)}:=\min_{1\le i\le N}\{\inf\{t\ge0:\,X_t^{i}\le0\}\}\wedge\min_{1\le i\le N}\{\inf\{t\ge0:\,Y_t^{i}\le0\}.
    \]
    On $[0,\sigma_{(1)})$ all coordinates are strictly positive, hence both $X$ and $Y$ solve the same non-absorbed SDE on the full index set $\{1,\dots,N\}$. By Assumption \ref{ass:segment-uniq-diff} they coincide on $[0,\sigma_{(1)})$, and therefore hit zero at the same time and in the same block: the sets of indices absorbed at $\sigma_{(1)}$ agree for $X$ and $Y$.

    Proceed by induction. Suppose for some $k\ge1$ that $X$ and $Y$ coincide on $[0,\sigma_{(k)})$ and that the alive sets $A_k$ coincide at time $\sigma_{(k)}$. On $[\sigma_{(k)},\sigma_{(k+1)})$ both processes solve the same reduced non-absorbed SDE on $A_k$ with the same initial conditions and Brownian increments (shifted at $\sigma_{(k)}$). By Assumption \ref{ass:segment-uniq-diff}, they coincide on this interval, and in particular have the same next exit time and exit block. Iterating up to exhaustion of indices yields $X\equiv Y$ on $[0,T]$.

    Thus the absorbed system enjoys pathwise uniqueness. 
    \qed
    \end{proof}

\section{Proof of Theorem \ref{thm:p-s}}\label{sec:thm1}
We split this section into two parts having already proved the weak well-posedness in \ref{thm:strong-system}.In the first part we show the tightness of the empirical measures. Then, in the second part, we show Theorem \ref{thm:p-s}-2). The last claim in  Theorem \ref{thm:p-s} simply follows from uniqueness of solutions to~\eqref{eq:lim} that is, under more regularity on $b$, guaranteed by Theorem \ref{thm:lim}. 

\subsection{Tightness of $\{\mu^{N}\}_{\N \geq 1}$ and Partial Girsanov transforms}
Let us first show  the tightness of $\{\mu^{N}\}_{\N \geq 1}$.
\begin{Lemma}
\label{Tightness:crit}
Suppose that $\mu^N_0=\frac1N\sum_{i=1}^N \delta_{Z^{N,i}}$ converges in probability to some deterministic measure $\rho\in \Pp(\R_+)$. In addition, let Assumption~\ref{ass:1}-(i) hold. Then, the sequence of random measures $\{\mu^{N}\}_{\N \geq 1}$ is tight in $\mathcal{P}(\Omega_T)$ under $\Q^N$. 
\end{Lemma}
\begin{proof}
By \cite[Prop. 2.2-ii]{Sznitman}, its tightness 
results from the tightness of the
intensity
measure $\{\E_{\Q^N} \mu^N(\cdot)\}_{N\ge 1}$.
By symmetry, it suffices to check the
tightness of $\{\Lc(X^{N,1}_t: 0\le t\le T)\}_{N\ge 1}$. The boundedness of $b$ yields the existence of some constant $C>0$ such that
$$\E_{\Q^N}[|X_t^{N,1}-X_s^{N,1}|^4]  \leq C(t-s)^4 +  C\E_{\Q^{N}}\left[\left|\int_s^t\iota(X_u^{N,1})\d W_u^1\right|^4\right] \le C(t-s)^4 + C(t-s)^2,$$
which fulfills the proof. \qed 
\end{proof}
We notice here that the fact that $b$ is bounded does not only play an important role when showing  tightness, but it will also be crucial for the identification of the limit as it enables us to use the Girsanov transform. The fact that particles behave as stopped Brownian motions (up to a Girsanov transform) will be used in what follows. However, passing from the interacting system to the system of independent stopped Brownian motions has a cost that was not uniform in $N$ (see the existence section) and as such a full Girsanov transform is not useful when we pass to the limit as $N\to \infty$.      
Nevertheless we can use that the individual behaviours are the ones of stopped Brownian motions by introducing 
 \textit{Partial transforms} related to \eqref{eq:p-s}. 
 
For a fixed $r<N$ that corresponds to $r$ particles that are transformed to independent stopped Brownian motions and whose influence is  removed from the other $N-r$ particles. The cost of such transforms will be uniform in $N$. Although for the identification of the limit we will only use the case $r=1$, we give them below for an arbitrary $r<N$.

Let us choose, for simplicity, the first  $r$ particles to be transformed, and consider the following reference system defined on the probability space  $(E,\Ec, \G=(\Gc_t)_{t\ge 0},\W^{(r,N)})$:
\begin{equation*}
    \begin{cases}
     \displaystyle \hat{X}_t^{N,i}= Z^{N,i} + \int_0^t\iota(\hat{X}^{N,i}_s)  \d\hat W_s^i , \quad i\leq r, t\leq T \\
       \displaystyle \hat{X}_t^{N,i}= Z^{N,i} + \int_0^t\iota(\hat{X}^{N,i}_s)  \d\hat W_s^i + \iota(\hat{X}^{N,i}_s)\frac{1}{N}\sum_{j=r+1}^N \int_0^t b(s,\hat{X}_s^{N,i},\hat{X}_s^{N,j}) \iota(\hat{X}^{N,j}_s) \d s , \quad r+1\leq i\leq N, t\leq T.
    \end{cases}
\end{equation*}
Obviously, it is well-posed as it can be obtained, as in the previous section, from $\overline X$. We now  study the change of measure between $\hat X$ and $X$. The corresponding drift vector is  
\begin{equation*}
 \beta^{(r)}_t(x):= \Big(b_t^{N,1}(x), \dots, b_t^{N,r}(x), \frac{1}{N}
 \sum_{i=1}^r b(t, x^{r+1},x^{i})\iota(x^i), \dots,
\frac{1}{N} \sum_{i=1}^r  b(t, x^{N},x^{i})\iota(x^i) \Big) .
\end{equation*}
In the sequel we will need uniform w.r.t~$N$ bounds for moments of
\begin{equation} \label{def:Z^(k)}
Z_T^{(r)}:= \exp\left\{ \int_0^T \beta^{(r)}_t(\widehat{X}_t) \cdot
dW_t -\frac{1}{2}\int_0^T| \beta^{(r)}_t(\widehat{X}_t)|^2 dt \right\}.
\end{equation}

\begin{Proposition}
\label{GirsanovKparticles}
For any $T>0$, $\gamma>0$ and $r\geq 1$ 
there exists 
$ C(T, \gamma, r)$ s.t.
\begin{equation*}
\forall N\geq 1, \quad  \E_{\W^{(r,N)}} \exp\left\{ \gamma
\int_0^T |\beta^{(r)}_t(\widehat{X}_t)|^2dt\right\}\leq C(T, \gamma, r).
\end{equation*}
\end{Proposition}
\begin{proof}
    Observe that 
    $$ |\beta^{(r)}_t(x)|^2= \sum_{i=1}^r |b_t^{N,i}(x)|^2 + \frac{1}{N^2} \sum_{j=r+1}^N \Big( \sum_{i=1}^r b(t, x^{j+1},x^i) \iota(x^i)\Big)^2\leq r\|b\|^2_\infty + \frac{(N-r)r^2}{N^2}\|b\|^2_\infty $$
    This yields the desired result with $C(T,\gamma, r)= e^{2\gamma Tr\|b\|^2_\infty}$. \qed 
\end{proof}

\subsection{Identification of the limit}\label{ssec:identification}

Now we are ready to identify the limit. Having showed that $\{\mu^{N}\}_{N \geq 1}$ is tight,  we may extract a convergent subsequence, which is still denoted by $\{\mu^{N}\}_{N \geq 1}$, of limit $\mu$. Since $\mu^N \to \mu$ in law in $\mathcal{P}(\Omega)$, we have that $\mu^N\times \mu^N \to \mu \times\mu $.

It  remains to prove that $\mu$
is a weak solution to \eqref{eq:lim}.   
 For any $\varphi\in C^2_c(\R)$, the process $(M_t^\varphi)_{t\geq 0}$, defined in the martingale problem, should be a $\mu$ martingale. To this end, it suffices to show that for all $t>s>0$, all continuous bounded function $\Phi\in C(\Omega_s)$, we have $\Psi(\mu)=0$ a.s., where
for $\Q \in \Pp(\Omega)$,
\be* 
\Psi(\Q):=
\E^\Q\left[\Phi((F_r)_{r\in [0,s]})
\Big(\varphi(F_t)-\varphi(F_s)-
\int_s^t\iota(F_u) \Big[  \varphi''(F_u)+ \varphi'(F_u)\int_{\R}  b(u, F_u,y)   \iota(y)  \Q_u(\d y)  \Big]\d u \Big)\right], 
\ee* 
We observe that for any $\Q\in \Pp(\Omega)$, it holds that $\Psi(\Q)=\Theta(\Q\otimes \Q)$, where we define for $\Pi\in \Pp(\Omega^2)$
\begin{align*}
\Theta(\Pi):=\int_{\Omega^2}\!\!\!\!\!
\Phi((x_r)_{r\in [0,s]})
\Big(\varphi(x_t)-\varphi(x_s)-
\int_s^t \iota(x_u)\Big[  \varphi''(x_u)  + \varphi'(x_u)b(u, x_u,y_u)\iota(y_u)\Big]\d u \Big) \Pi(\d x,\d y).
\end{align*}
We separate the rest of the proof in several steps.

{\it Step 2.1.} Here we show that for some constant $A$, for all $N\geq 2$,
\begin{equation}\label{scc1}
\E \Big[ [\Theta(\mu^N\times \mu^N)]^2  \Big] \leq \frac {A} N.
\end{equation}
Applying Itô's formula on $\varphi(X_t^{N,i})$ and summing in $N$ we have that
\be* 
&& \frac{1}{N} \sum_{i=1}^N\Big[\varphi(X_t^{N,i}) - \varphi(Z^{N,i}) -\int_0^t \iota(X_s^{N,i})\Big[\varphi''(X_s^{N,i})- \varphi'(X_s^{N,i}) \sum_{j=1}^N \frac{1}{N} b (s, X_s^{N,i},X_s^{N,j}) \iota(X_s^{N,j})\d s \Big ] \\
&=&\frac{1}{N} \sum_{i=1}^N\int_0^t \iota(X_s^{N,i})\varphi'(X_s^{N,i})  \d W_s^i.
\ee* 
We recognise in the left hand side of the above expression parts of $\Theta(\mu^N\times \mu^N)$.
Hence
$$\E \Big[ [\Theta(\mu^N\times \mu^N)]^2  \Big] \leq \frac{T\|\varphi'\|^2_\infty \|\Phi\|_\infty^2}{N}.$$
{\it Step 2.2.} for any $\eta \in (0,1]$, we define a regularised version of the function $\iota$
by $f_\eta$, e.g. 
\be*
f_\eta := \min\left(1, C_\eta\left(\frac{x^3}{3}+\frac{x^5}{5}-\frac{x^4}{2}\right)^+\right), \quad \mbox{with } C_\eta:=\frac{30}{10\eta + 6\eta^5 - 15\eta^3}.
\ee* 
It is straightforward to check that $f_\eta$ is twice continuously differentiable on $\R$ and takes values in $[0,1]$.
Accordingly, we define $\Theta_\eta$. We have
\begin{align*}
\E[|\Psi(\mu)|]
\leq & \E[|\Theta(\mu\otimes\mu)-\Theta_\eta(\mu\otimes \mu)|]+ \limsup_N| \E[|\Theta_\eta(\mu\times \mu)|]- \E[|\Theta_\eta(\mu^N\otimes \mu^N)|]|\\
&+ \limsup_N \E[|\Theta_\eta (\mu^N\otimes\mu^N)-\Theta(\mu^N\otimes\mu^N)|] + \limsup_N \E[|\Theta(\mu^N\otimes \mu^N)|].
\end{align*}
By {\it Step 2.1.} the last term converges to zero as $N\to \infty$. The second term converges to zero by weak convergence and the fact that $\Theta_\eta$ is continuous and bounded (this is where the continuity in space of $b$ plays a role).  In the next step we will prove that the remaining terms are uniformly small w.r.t. $\eta$.

{\it Step 2.3.} Let us start with

$$I_\eta:= \limsup_N \E[|\Theta_\eta (\mu^N\otimes\mu^N)-\Theta(\mu^N\otimes\mu^N)|] .  $$

Observe that 
\begin{align*}
    |\Theta_\eta (\mu^N\otimes\mu^N)-\Theta(\mu^N\otimes\mu^N)|&\leq \phi_\infty  \varphi''_\infty \int_s^t \frac{1}{N} \sum  |\iota(X_u^{N,i}) - f_\eta(X_u^{N,i}) | \d u \\
    &  +\|\phi\|_\infty \|\varphi'\|_\infty \|b\|_\infty \int_s^t \frac{1}{N^2} \sum_{i,j}   |\iota(X_u^{N,i})\iota(X_u^{N,j}) - f_\eta(X_u^{N,i}) f_\eta(X_u^{N,j})| \d u\\
    &\leq A  \int_s^t \frac{1}{N} \sum_{i=1}^N  |\iota(X_u^{N,i}) - f_\eta(X_u^{N,i}) | \d u ,
\end{align*}
for some $A>0$ that may change from line to line.
 This leads by exchangabilty to 
  \be* 
  \E[|\Theta_\eta (\mu^N\otimes\mu^N)-\Theta(\mu^N\otimes\mu^N)|]\leq  A \int_s^t \E|\iota(X_u^{N,1}) - f(X_u^{N,1}) | \d u \leq A\int_s^t \Q [0<X_u^{N,1}<\eta] \d u.
  \ee* 
 Now, we use the partial transformation to  transform $X^{1,N}$ to a stopped Brownian motion. This leads to 
\be* 
\E[|\Theta_\eta (\mu^N\otimes\mu^N)-\Theta(\mu^N\otimes\mu^N)|]\leq  A\int_s^t \E_{\Q^{(1,N)}}[Z_T^{(1)}\1_{\{X_u^{N,1}\in(0, \eta)\}}] \d u.
\ee* 
Now, using Proposition \ref{GirsanovKparticles} for $r=1$, we get
\be* 
\E[|\Theta_\eta (\mu^N\otimes\mu^N)-\Theta(\mu^N\otimes\mu^N)|]\leq  A\int_s^t \P[0<W_{\tau \wedge u }<\eta]^{1/2}\d u
\ee* 
This leads us to the control of $\P[W_{\tau \wedge u } \in (0, \eta)]$. The latter converges,  for every fixed $u>0$, to zero as $\eta \to 0$. By dominated convergence we can conclude the desired result. 

For the term that  remains, we use Fatou's lemma and the above computation.
\be*
    \E[|\Theta_\eta (\mu\otimes\mu)-\Theta(\mu\otimes\mu)|] &\leq & C \int_s^t \int_{\Omega^2} \1_{\{f(u) \in (0, \eta)\}}\d u \mu\otimes \mu (\d f,\d y)\\ 
    &\leq &  C\lim_N  \int _{\Omega}\int_s^t  \1_{\{f(u) \in (0, \eta)\}}\d u\mu^N (\d f)
\ee* 
Repeat the above Partial Girsanov transform argument to conclude the proof of 3.

Finally, note that the subset $\{f\in\Omega: f(t)\ge 0, \forall t\in\R_+\}$ is closed. Applying the Portmanteau theorem, we may conclude  $\Q[\{f\in\Omega: f(t)\ge 0, \forall t\in\R_+\}]\ge \limsup_{N\to\infty}\Q^N[\{f\in\Omega: f(t)\ge 0, \forall t\in\R_+\}]=1$ and thus $\Q[F_t\ge 0,~ \forall t\ge 0]=1$. 

Next, we show in Section \ref{sec:thm2} that the mean-filed SDE \eqref{eq:lim} admits    a unique strong solution which is clearly a weak solution. Combined with Theorem \ref{thm:p-s}-2), this yields the propagation of chaos and thus proves 
 Theorem \ref{thm:lim}-3).

\section{Proof of Theorem \ref{thm:lim}}\label{sec:thm2}
We prove in this section the well-posedness of~\eqref{eq:lim}. It is worth noting that, thanks to Section~\ref{ssec:identification}, there exists at least a weak solution to \eqref{eq:lim} under the hypothesis of Theorem~\ref{thm:p-s}-2). Nevertheless, to obtain the uniqueness result for \eqref{eq:lim}, we need more subtle properties for its solution, which requires more conditions on the coefficient $b$. This yields, as a by-product,  the existence of a (strong) solution to \eqref{eq:lim}. Therefore, we combine in the first part of Section \ref{sec:thm2} the derivation of the desired properties and the existence of solutions to \eqref{eq:lim} using a fixed-point argument. The second part is devoted to the uniqueness. In the last subsections we give proofs of some technical lemmas.

Throughout this section we fix an arbitrary time horizon $T>0$. 

\subsection{Proof of the existence}\label{ssec:existence}
We first deal with the existence. Let $\Pc\subset \Pc(\R_+)$ be the set of probability measures on $\R_+$ of finite first order moment. Define the set of probability flows
\be* 
\Pc_T:=\big\{\mu=(\mu_t)_{0\le t\le T}:~ \mu_t\in\Pc\big\}
\ee* 
and recall $\Omega_T$ is the space of continuous functions on $[0,T]$. We endow $\Pc_{T}$ with the metric $\Wc_T(\mu,\nu):=\sup_{0\le t\le T}  \Wc(\mu_t,\nu_t)$ and endow $\Omega_T$ with the uniform norm $\|\cdot \|_T$. Let $d_T$ be the distance on $\Pc_{T}\times \Omega_T$ defined by $d_T((\mu,f),(\nu,g)):=\Wc_T(\mu,\nu)+\| f -g\|_T$. Rewriting \eqref{eq:lim} in differential form, one has 
\be* 
\d X_t =  \iota(X_t)\left[ \left(\int_{(0,\infty)} b(t,X_t,y)\mu_t(\d y)\right)\d t +  \sqrt{2}\d W_t \right] 
= \1_{\{\tau>t\}}\left[ B\big(t,X_t,\mu_t,\mu_t(\R_+^*)\big)\d t +  \sqrt{2}\d W_t \right], 
\ee* 
where $\mu_t:=\Lc(X_t)$, $\tau:=\inf\{t\ge 0:~ X_t\le 0 \}$, $\R_+^*:=(0,\infty)$ and $B: \R_+\times\R\times\Pc\times [0,1]\to\R$ is given as
\be* 
B(t,x,\lambda,a):=\int_{\R_+}b(t,x,y)\lambda(\d y) - b(t,x,0)(1-a). 
\ee* 
Let us introduce the operator $\Gamma$ on $\Pc_T\times C_T$ defined by
\be* 
\Gamma(\mu,f):=\Big(\big(\Lc(Y^{\mu,f}_{t\wedge \tau^{\mu,f}})\big)_{0\le t\le T},~ \alpha^{\mu,f}\Big),
\ee* 
where $\alpha^{\mu,f}(t):=\P[\tau^{\mu,f}> t]$, $\tau^{\mu,f}:=\inf\{t\ge 0:~ Y^{\mu,f}_t\le 0\}$ and 
\beq\label{sde-freeze} 
Y^{\mu,f}_t = Z + \int_0^t B(s,Y^{\mu,f}_s,\mu_s,f(s))\d s + \sqrt{2}W_t,\quad \forall t\in [0,T].
\eeq  
For any $(\mu,f)\in \Pc_T\times C_T$, the SDE \eqref{sde-freeze} has a unique solution and thus $\Gamma$ is well defined. Next we point out that any fixed point of $\Gamma$ allows us to construct a solution to \eqref{eq:lim}. Namely, let $(\mu,f)$ be a fixed point of $\Gamma$, i.e. $\Lc(Y^{\mu,f}_{t\wedge \tau^{\mu,f}})=\mu_t$ and $\alpha^{\mu,f}(t)=\P[\tau^{\mu,f}> t]=f(t)$. Then a straightforward verification yields
\be* 
\d Y^{\mu,f}_{t\wedge \tau^{\mu,f}} &=& \1_{\{\tau^{\mu,f}>t\}}\left[ B(t,Y^{\mu,f}_t,\mu_t,f(t))\d t+ \sqrt{2}\d W_t\right] \\
&=&\iota\big(Y^{\mu,f}_{t\wedge \tau^{\mu,f}}\big)\left[ \left(\int_{(0,\infty)} b(t,Y^{\mu,f}_{t\wedge \tau^{\mu,f}},y)\mu_t(\d y)\right)\d t + \sqrt{2}\d W_t\right],\quad \forall t\in[0,T],
\ee* 
which implies that  $(Y^{\mu,f}_{t\wedge \tau^{\mu,f}})_{0\le t\le T}$ solves \eqref{eq:lim} up to time $T$. Therefore, the proposition below ensures a  fixed point of $\Gamma$ and thus a  solution to \eqref{eq:lim} on $[0,T]$.  For each $L>0$, set
\be*
\Pc_{T}^L &:=& \left\{\nu\in\Pc_T:~ \sup_{0\le s<t\le T} \left(\int_{\R_+} x^2 \nu_t(\d x) + \frac{\Wc(\nu_t,\nu_s)}{|t-s|^{1/6}}\right)\le L\right\}  \\ \Omega_T^L &:=& \left\{f\in \Omega_T: ~ \sup_{0\le s<t\le T}  \frac{|f(t)-f(s)|}{|t-s|^{1/6}}\le L\right\}.
\ee*
Then we have the following proposition.
\begin{Proposition}\label{prop:fixed-point}
Under suitable conditions, $\Gamma$ has a fixed point in $\Pc_T\times \Omega_T$.
\end{Proposition}
\begin{proof}
The proof of Proposition \ref{prop:fixed-point} is a combination of the lemmas below. More precisely, Lemma \ref{lem:compact} shows that $\Gamma(\Pc^L_T\times \Omega^L_T)\subset\Gamma(\Pc_T\times \Omega_T)\subset \Pc^L_T\times \Omega^L_T$ for some $L>0$ large enough. As $\Pc^L_T\times \Omega^L_T$ is compact and $\Gamma$ is continuous with respect to $d_T$ in view of Lemmas \ref{lem:compact} and \ref{lem:continuity}, Schauder's fixed-point theorem allows to show the existence of $(\mu,f)\in \Pc^L_T\times \Omega^L_T$ such that $\Gamma(\mu,f)=(\mu,f)$. \qed
\end{proof}
Let $\Phi: \R\to [0,1]$ the cumulative distribution function of a standard normal distribution, i.e. 
\be* 
\Phi(x):=\int_{-\infty}^x \frac{1}{\sqrt{2\pi}}e^{-y^2/2}\d y,\quad \forall x\in\R. 
\ee* 
\begin{Lemma}\label{lem:compact}
 $\Pc^L_{T}\times \Omega^L_T$ is $d_T-$compact for every $L>0$.
\end{Lemma}
\begin{proof}
The compactness of $\Omega^L_T$ is straightforward by Arzel\`a-Ascoli's theorem.  It remains to treat $\Pc^L_T$. Pick an arbitrary sequence $(\nu^n)_{n\ge 1}\subset \Pc^L_T$ and let us show that it has a convergent subsequence of limit belonging to $\Pc^L_T$. For each $t\in [0,T]$, note that there exists a weakly convergent subsequence of  $(\nu_t^n)_{n\ge 1}$ as it is uniformly integrable. Using the diagonal arguments, one may extract some subsequence, still denoted by $(\nu^n)_{n\ge 1}$ for the sake of simplicity, such that $(\nu_t^n)_{n\ge 1}$ converges weakly for each $t\in [0,T]\cap\Q$. Combined with the fact
\be*
\int_{\R_+}x^2\nu_t^n(\d x) \le L,\quad \forall t\in [0,T],~ n\ge 1, 
\ee*
the above weak convergence is indeed the convergence under $\Wc$, i.e. $\lim_{m, n\to\infty}\Wc(\nu^m_t,\nu^n_t)=0$ for all $t\in [0,T]\cap\Q$. Next, for each $t\in [0,T]$, take a sequence $(t_k)_{k\ge 1}\subset [0,T]\cap\Q$ converging to $t$. For every $\eps>0$, one has
\be*
\lim_{m,n\to\infty} \Wc(\nu^m_t,\nu^n_t) \le \lim_{m,n\to\infty} \big(\Wc(\nu^m_{t_k},\nu^m_{t}) + \Wc(\nu^m_{t_k},\nu^n_{t_k}) + \Wc(\nu^n_{t_k},\nu^n_{t})\big) \le 2L|t_k-t|^{1/6},
\ee*
which shows that $(\nu^n_t)_{n\ge 1}$ is a Cauchy sequence and thus converges under $\Wc$. Denote $\nu\equiv (\nu_t:=\lim_{n\to\infty}\nu^n_t)_{0\le t\le T}\in \Pc_T$. One deduces thus 
 $\Wc(\nu_t,\nu_s) \le L|t-s|^{1/6}$ for all $t,s \in [0,T]$ and further  $\lim_{n\to\infty} \Wc_T(\nu^n,\nu)=0$. 
Finally, by Fatou's lemma, one has for every $t\in [0,T]$
\be*
\int_{\R_+} x^2\nu_t(\d x) \le \liminf_{n\to\infty} \int_{\R_+} x^2 \nu^n_t(\d x) \le L,
\ee*
which fulfills the proof. \qed
\end{proof}
\begin{Lemma}\label{lem:image}
There exists some constant $L>0$ such that $\Gamma(\Pc_T\times \Omega_T) \subset \Pc^L_{T}\times \Omega^L_T$.
\end{Lemma}
\begin{proof}
Take an arbitrary $(\nu,f)\in \Pc_T\times \Omega_T$. For  simplicity we drop the superscript $(\nu,f)$ without any danger of confusion, i.e. $Y\equiv Y^{\nu,f}$,  $\tau\equiv \tau^{\nu,f}$, etc. Denote further
$B_t\equiv B(t,Y_t, \nu_t, f(t))$. 
For any $0\le t<t+\Delta t\le T$, let us estimate respectively $\Wc(\Lc(Y_{\tau\wedge t}),\Lc(Y_{\tau\wedge (t+\Delta t)}))$ and $\alpha(t)-\alpha(t+\Delta t)$. By definition, one has
\be*
\Wc(\Lc(Y_{\tau\wedge t}),\Lc(Y_{\tau\wedge (t+\Delta t)})) \le \E[|Y_{\tau\wedge (t+\Delta t)} -Y_{\tau\wedge t}|] =  \E\left[\left|\int_{\tau\wedge t}^{\tau\wedge (t+\Delta t)} B_s\d s+\sqrt{2}\d W_s\right|\right] \le C\sqrt{\Delta t}.
\ee*
Hereafter, $C>0$ is used to denote generic  constants that may change from line to line, but depends only on the designated variables. 
Applying Girsanov's theorem with 
\be*
\frac{\d\Q}{\d\P}\Big|_t :=  \exp\left(-\int_0^t \lambda_s\d W_s-\frac{1}{2}\int_0^t\lambda_s^2\d s\right),\quad \mbox{where } \lambda_s := \frac{B_s}{\sqrt{2}},
\ee*
$(W^{\Q}_t := W_t+\int_0^t \lambda_s\d s)_{0\le t\le T}$ is a Brownian motion under $\Q$ and one has 
\be*
\alpha(t)-\alpha(t+\Delta t) &=& \P[\tau> t]- \P[\tau> t+\Delta t] \\
&=& \E^{\Q}\left[\left(\frac{\d\P}{\d\Q}\right)_{t+\Delta t} \left(\iota\left(\inf_{0\le r\le t}(Z+\int_0^r\sqrt{2} \d W^{\Q}_s)\right) -\iota\left(\inf_{0\le r\le t+\Delta t}(Z+\int_0^r\sqrt{2} \d W^{\Q}_s)\right)  \right)\right] \\
&\le&\exp\left(\frac{3}{2}\int_0^{T}\lambda_s^2ds\right)\left(\Q\left[\inf_{0\le r\le t}(Z+\int_0^r\sqrt{2} \d W^{\Q}_s)>0\right] -  \Q\left[\inf_{0\le r\le t+\Delta t}(Z+\int_0^r\sqrt{2} \d W^{\Q}_s)>0\right]\right)^{1/2}\\
&\le& C\left(\int_{\R_+}\left(\Phi\left(x/\sqrt{2t}\right)-\Phi\left(x/\sqrt{2(t+\Delta t)}\right)\right)\rho(\d x)\right)^{1/2} ~~\le~~ C\Delta t^{1/6},
\ee*
where the first inequality follows from H\"older inequality. 
Finally, we see the fact that
\be*
\E[|Y_{\tau\wedge t}|^2] 
\le  3\left(\E[|Z|^2]+ \E\left[\left|\int_{0}^{T} |B_s|\d s\right|^2\right] + \E\left[\int_{0}^{T}2\d s\right]\right) \le C,
\ee*
which allows to conclude the proof. \qed
\end{proof}
\begin{Lemma}\label{lem:continuity}
$\Gamma: \Pc_T\times \Omega_T\to\Pc_T\times \Omega_T$ is $d_T-$continuous.
\end{Lemma}
\begin{proof}
Let $(\nu^n,f^n)\to(\nu,f)$. Again,  we denote for simplicity 
\be*
Y^n\equiv Y^{\nu^n,f^n} && Y\equiv Y^{\nu,f} \\
\tau_n\equiv \tau^{\nu^n,f^n} && \tau\equiv\tau^{\nu,f} \\
B^n_t\equiv B(t,Y^n_t, \nu^n_t,f^n(t)) && B_t\equiv B(t,Y_t, \nu_t,f(t)),~ \mbox{etc.}
\ee*
We estimate first $\E[\sup_{0\le t\le T}|Y^{n}_t-Y_t|^2]$ following the classic arguments. For each $t\in [0,T]$, one has
\be*
Y^n_t-Y_t = \int_0^t\left(B\big(u,Y^n_u, \nu^n_u,f^n(u)\big)-B\big(u,Y_u,\nu_u,f(u)\big) \right) \d u.
\ee*
By the Lipschitz continuity of $B$ and the inequality
$\max\big(\Wc(\nu^n_u,\nu_u), |f^n(u)-f(u)|\big) \le d_T((\nu^n, f^n), (\nu,f)\big)$,  there exists $C>0$ large enough such that
\be*
\E[\sup_{0\le s\le t}|Y^{n}_s-Y_s|^2] \le C\left(d_T((\nu^n, f^n), (\nu,f)\big) + \int_0^t  \E[\sup_{0\le s\le u}|Y^{n}_s-Y_s|^2] \d u\right),\quad \forall t\in [0,T]. 
\ee*
This, combined with Gronwall's inequality, yields 
$\lim_{n\to\infty}\E[\sup_{0\le t\le T}|Y^{n}_t-Y_t|^2] =0$. In particular, $\sup_{0\le t\le T}|Y^{n}_t-Y_t|$ converges to zero in probability. Let $\varepsilon > 0$ be arbitrary. The continuity of $Y$ implies
 \be*
\{\tau > t\} = \bigcup_{k \in \mathbb N} A_k := \bigcup_{k \in \mathbb N} \{Y_s \geq 1/k, \ \forall 0\le s\le t\}.
\ee*
Hence, we have that for some $k_0 \in \N$ that $\P[A_{k_0}] > \P[\tau > t] -\eps/2$. Since $Y^{n} \to Y$ uniformly in probability, there exists $N$ such that
\be*
\P[\sup_{s \in [0, t]} | Y_s^{n} - Y_s|
 \geq 1/k_0] \le  \frac{\varepsilon}{2}
 \ee*
holds for all $n > N$.  We obtain thus
\be*
\P[\tau_n > t] \ge  \P[A_{k_0} \cup \{\sup_{s \in [0, t]} | Y_s^{n} - Y_s| >   1/k_0  \}^c] \ge \P[\tau > t] - \varepsilon.
\ee*
The reverse inequality can be argued similarly, where we note that $\P[\tau \le t]=\P[\inf_{0\le s\le t}Y_s\le 0]=\P[\inf_{0\le s\le t}Y_s<0]$ as $Y$ is a drifted Brownian motion,  and  we can rewrite  
\be*
\{\inf_{0\le s\le t}Y_s<0\}=\bigcup_{k \in \mathbb N} \{\inf_{s \in [0, t]} Y_s \leq -1/k\}.
\ee*
Thus with a similar calculation, we obtain $\lim_{n\to\infty}\P[\tau_n\le  t] \ge  \P[\tau \le t]-\eps$ and thus $\lim_{n\to\infty}\alpha^n(t)=\lim_{n\to\infty}\P[\tau_n> t] =  \P[\tau> t]=\alpha(t)$. 
 Next, we compute $\Wc(\Lc(Y^n_{\tau_n\wedge t}),\Lc(Y_{\tau\wedge t}))$. By definition, one has $\Wc(\Lc(Y^n_{\tau_n\wedge t}),\Lc(Y_{\tau\wedge t}))\le  \E[|Y_{\tau_n \wedge t}^n - Y_{\tau \wedge t}|]$. Further, 
\be* 
\E[|Y_{\tau_n \wedge t}^n - Y_{\tau \wedge t}|] &=&   \E[|Y_{t}^n - Y_{t}|{\1}_{\{\tau_n>t, \tau>t\}}] + \E[|Y_{t}^n |{\1}_{\{\tau_n>t, \tau\le t\}}] +  \E[|Y_{t}|{\1}_{\{\tau_n\le t, \tau>t\}}] \\
&\le &\E[|Y_{t}^n - Y_{t}|] + \E[|Y_{t}^n|^2]^{1/2}\P[\tau_n>t, \tau\le t]^{1/2} +   \E[|Y_{t}|^2]^{1/2}\P[\tau_n\le t, \tau> t]^{1/2}. 
\ee*
Repeating the above reasoning, one may show $\lim_{n\to\infty}\P[\tau_n> t, \tau\le  t]=0=\lim_{n\to\infty}\P[\tau_n\le t, \tau> t]$, which yields
 $\lim_{n\to\infty} \Wc(\Lc(Y^n_{\tau_n\wedge t}),\Lc(Y_{\tau\wedge t}))=0$.  Having shown
 \be*
\lim_{n\to\infty} \Wc(\Lc(Y^n_{\tau_n\wedge t}),\Lc(Y_{\tau\wedge t}))=0 &\mbox{and}& \lim_{n\to\infty} \alpha^n(t)=\alpha(t), \quad \forall t\in [0,T], 
 \ee*
 the uniform continuity in Lemma  \ref{lem:image} finally shows
$\lim_{n\to\infty} d_T(\Gamma(\nu^n,f^n),\Gamma(\nu,f))=0$.\qed
\end{proof}

\subsection{Proof of the uniqueness }\label{ssec:uniqueness}

This section is reserved for showing the uniqueness of solution to \eqref{eq:lim} on $[0,T]$. We fix an arbitrary solution $X$ to \eqref{eq:lim}. We first show the one-to-one relation between $X$ and the  parametric nonlinear Fokker-Planck equation below. 
\begin{Proposition}
The uniqueness of $X$ is equivalent to the uniqueness of $u:\R_+^2\to\R$ satisfying
\label{prop:equiv}
\beq\label{eq:FP}
\begin{split}
 & \partial_t u(t,x) =\partial^2_{xx}\big(u(t,x)\big)  - \partial_x\big( B^{\nu,\alpha}(t,x)u(t,x)\big) ,\quad \forall t, x>0 \\
 & u(0,x)=\rho(x),\quad u(t,0)=0 ,\quad \forall t, x>0 \\
 & \beta(t)=\int_0^\infty u(t,y)\d y,\quad \forall t>0 \\
 & \nu_t(\d x) = \big(1-\beta(t)\big )\delta_0(\d x) + u(t,x)\d x,\quad \forall t>0 \\
 & B^{\nu,\beta}(t,x) =B(t,x,\nu_t,\beta(t)),\quad \forall t, x>0.
\end{split}
\eeq 
\end{Proposition}
Proposition \ref{prop:equiv} is the basis to adopt the fixed-point argument, and we need the following auxiliary result to prove it. Denote $(\mu_t)_{t\geq 0}:=(\Lc(X_t))_{t\geq 0}$, $\tau:=\inf \{t\ge 0: X_t\le 0\}$  and  $\alpha(t):=\P[\tau> t]$. For $(\nu,f)\in\Pc_T\times \Omega_T$, define the function $B^{\nu,f}:\R_+\times\R\to\R$ by $B^{\nu,f}(t,x):=B(t,x,\nu_t,f(t))$. Then we have  the following proposition.
\begin{Proposition}\label{prop:propertiesMarginals}
 Under the hypothesis of Theorem~\ref{thm:lim} and writting $B\equiv B^{\mu,\alpha}$, we have the following results. 
    \begin{enumerate}[label=\upshape(\roman*),ref= (\roman*)]
    \item \label{lem:decom}  For each $t>0$, there exists a sub-probability density $p(t,\cdot)$ supported on $(0,\infty)$ such that
\be*
\mu_t(\d x) = \big(1-\alpha(t)\big )\delta_0(\d x) + p(t,x)\d x &\mbox{and}& \alpha(t) = \int_0^{\infty}p(t,x)\d x. 
\ee*

    \item\label{th:second} The family $(\mu_t)_{t\geq0}$ is the unique weak solution to the following  Fokker-Planck equation, i.e. 
    \beq \label{eq:weakFP}
\frac{\d}{\d t}\int_{\R}f(x)\mu_t(\d x)  = \int_{(0,\infty)}\left[ B(t,x) f'(x) + f''(x)\right] \mu_t(\d x),\quad \forall t\in (0,T], 
\eeq
 holds for all smooth functions $f:\R\to\R$ of compact support. 

    \item\label{th:third} The function $p$ is a solution to \eqref{eq:FP}, with $\alpha$, $\mu$ and $B^{\mu,\alpha}$  defined above.
  \end{enumerate}
\end{Proposition}
\begin{proof}
By definition, one may rewrite 
\be*
\mu_t(\d x) = \big(1-\alpha(t)\big )\delta_0(\d x) + \iota(x)\mathbb P[X_t\in \d x]. 
\ee*
\rmi Define the SDE for $Y$ by
\be*
Y_t=Z + \int_0^t B(s,Y_s)\d s + \sqrt{2}W_t, \quad \forall t\ge 0.
\ee*
Note that $X_t=Y_t$ on the event $\{\tau>t\}$. So, for any Borel set $A\subseteq(0,\infty)$, we have 
\be*
\P[X_t\in A]=\P[Y_t\in A,\tau>t] \le \P[Y_t \in A],
\ee*
which implies that the distribution of $X_t$ restricted on $(0,\infty)$ is absolutely continuous with respect to  the distribution of $Y_t$. The latter admits a probability density as $B$ is bounded. 

\noindent \rmii One has by definition 
\be* 
B(t,x)=\int_{\R_+}b(t,x,y)\mu_t(\d y) - b(t,x,0)\big(1-\alpha(t)\big)=\E[b(t,x,X_t)] - b(t,x,0)\big(1-\alpha(t)\big),
\ee* 
which implies, by Lemma \ref{lem:image},  that $t\mapsto B^{\mu,\alpha}(t,\cdot)$ is H\"older continuous and $\partial_x B^{\mu,\alpha}, \partial^2_{xx} B^{\mu,\alpha}$ are bounded and Lipschitz in $x$. For the SDE 
\begin{equation}    
\label{eq:linearSDE}
Z_t = Z +\int_0^t \iota(Z_s)B(s,Z_s)\d s +\int_0^t \iota(Z_s)\sqrt{2}\d W_s,\quad \forall t\in [0,T] 
\end{equation}
$X$ is its solution by definition, where we use the fact 
that \eqref{eq:linearSDE} has a unique  solution if and only if the same holds for the corresponding SDE without absorption
\be*
\d Z_t = B(t,Z_t)\d t +\sqrt{2}\d W_t,\quad \forall t\in [0,T]. 
\ee*
This is ensured as $B^{\mu,\alpha}$ is Lipschitz in $x$. Therefore, $X$ is the unique  solution of \eqref{eq:linearSDE} and the family $(\mu_t)_{t\geq 0}$, by Lemma 2.3 of \cite{Figalli2008}, is the unique weak solution to the Fokker-Planck equation
\beq 
\label{eq:linearFPmeasure}
\frac{\d}{\d t}\int_{\R}f(x)\nu_t(\d x)  = \int_{(0,\infty)}\left[ B(t,x) f'(x) + f''(x)\right] \nu_t(\d x),\quad \forall t\in (0,T], f\in C_c^\infty(\R). 
\eeq
\rmiii To conclude, we apply Lemma~1.10 and Theorem 2.2 of Chapter VI in  \cite{GM1992}. Namely,  the following  (linear) Fokker-Planck equation on the half space
\beq 
\begin{split}
 & \partial_t u(t,x) =\partial^2_{xx}\big(u(t,x)\big)  - \partial_x\big( B(t,x)u(t,x)\big) ,\quad \forall t, x>0 \\
 & u(0,x)=\rho(x),\quad u(t,0)=0 ,\quad \forall t, x>0
\end{split}
\eeq 
has a unique classical solution $u$ satisfying for some $C>0$
\beq\label{eq:regularity}
|\partial_t u(t,x)|+|\partial_x u(t,x)|+|\partial^2_{xx} u(t,x)| \le \frac{C}{t}\int_0^\infty  \exp\left(-\frac{(x-y)^2}{Ct}\right) \rho(y)\d y,\quad \forall t,x>0.
\eeq
By integration by parts, one deduces  that 
\beq \label{eq:mu} 
\nu_t(dx) =  \left(1-\int_0^\infty u(t,x)\d x\right)\delta_0(\d x) + u(t,x)\d x ,\quad \forall t\in [0,T]
\eeq 
is also a weak solution to  \eqref{eq:linearFPmeasure}, and further by uniqueness $\nu=\mu$ and $u=p$. Hence the equation \eqref{eq:FP} holds for $p$ by noting $B\equiv B^{\mu,\alpha}$. \qed
\end{proof}
We are now ready to prove Proposition \ref{prop:equiv}. 

\begin{proof}[of Proposition \ref{prop:equiv}]
It is clear that for the given solution $X$, $p$ defined in the proof of Proposition \ref{prop:propertiesMarginals}  is a solution to \eqref{eq:FP}. On the other hand, for any solution $u$ to \eqref{eq:FP}, define $Z$ to be the solution of 
\be*
\d Z_t = B(t,Z_t,\nu_t,\beta(t))\d t +\sqrt{2}\d W_t,\quad \forall t\in [0,T]. 
\ee*
Then a straightforward verification yields $(Z_t\iota(Z_t))_{t\ge 0}$ is a solution to  \eqref{eq:lim}. \qed
\end{proof}

\vspace{1mm} 

It remains to prove the uniqueness of solutions to \eqref{eq:FP}. Introduce the SDE 
\be* 
\d Y_t = B^{\mu,\alpha}(t,Y_t)\d t + \sqrt{2}\d W_t,\quad \forall t\ge 0, 
\ee*  
and denote by $(Y^{t,x}_s)_{s\ge t}$ the solution satisfying $Y^{t,x}_t=x$. Let $g(t,x,s,\cdot)$ be the density function of $Y^{t,x}_s$ for $s>t$, i.e. $\P[Y^{t,x}_s\in \d y]= g(t,x,s,y)\d y$. It is known that, for any fixed $s>0$ and $y\in\R$, $g$ satisfies the  backward Kolmogorov equation 
\be*
&&\partial_t g(t,x,s,y) =- \partial^2_{xx} g(t,x,s,y) - B^{\mu,\alpha}(t,x)\partial_x g(t,x,s,y),\quad \forall (t,x)\in  [0,s)\times\R, \\
&& g(s,x,s,y) = \delta_y(x),\quad \forall x\in\R.
\ee*
Then $\alpha$ admits the following representation.
\begin{Lemma}\label{lem:representation}
Under the conditions of Theorem \ref{thm:lim}, one has for all $s>0$
\beq 
\frac{\alpha(s)}{2}=\int_0^\infty  \rho(x)\d x\int_0^\infty g(0,x,s,y) \d y  - \int_0^\infty g(0,0,s,y)\d y- 
\int_0^s \alpha(t) \d t \int_0^\infty \partial_tg(t,0,s,y) \d y.
\label{eq:alpha}
\eeq  
\end{Lemma}
\begin{proof}
 Integrating the Fokker-Planck equation  of \eqref{eq:FP} over $x\in (0,\infty)$, one obtains
\be*
\int_0^\infty \partial_t p(t,x)\d x = \int_0^\infty \partial^2_{xx} p(t,x)\d x - \int_0^\infty \partial_x\big(B^{\mu,\alpha}(t,x) p(t,x)\big)\d x,
\ee*
which yields by Fubini's theorem and integration by parts
\be*
\alpha'(t) = \frac{\d }{\d t}\int_0^\infty p(t,x)\d x = \int_0^\infty \partial_tp(t,x)\d x =\int_0^\infty \left[\partial^2_{xx} p(t,x) - \partial_x \big(B^{\mu,\alpha}(t,x) p(t,x)\big)\right]\d x= -\partial_{x}p(t,0),\quad \forall t>0.
\ee*
Further, compute for all $(t,x)\in (0,s)\times (0,\infty)$
\be*
\partial_t (pg) + \partial_x \big(B^{\mu,\alpha}pg\big) - \partial_x \left(\partial_{x} pg - p\partial_x g\right) 
=\left(\partial_t p -\partial^2_{xx}p + \partial_x(B^{\mu,\alpha} p) \right)g + p \left(\partial_t g+ \partial^2_{xx} g + B^{\mu,\alpha}\partial_x g\right) = 0,
\ee*
which yields by integrating over $(0,s)\times (0,\infty)$
\be*
\int_0^s\int_0^\infty  \partial_t (pg) \d t \d x + \int_0^s\int_0^\infty  \partial_x \big(B^{\mu,\alpha}(t,x)pg\big)\d t \d x - \int_0^s\int_0^\infty \partial_x \left(\partial_{x} pg - p\partial_x g\right)\d t \d x = 0. \ee* 
Again, applying Fubini's theorem combined with the initial and  boundary conditions, one has
\be*
 0 
&=& \int_0^\infty \big( p(s,x)\delta_y(x) - \rho(x)g(0,x,s,y)\big) \d x +  \int_0^s \partial_{x}p(t,0)g(t,0,s,y) \d t\\
&= &p(s,y) - \int_0^\infty  \rho(x)g(0,x,s,y) \d x -  \int_0^s \alpha'(t) g(t,0,s,y) \d t,  \quad \forall (s,y)\in (0,\infty)\times (0,\infty).
\ee* 
Integrating $y$ over $(0,\infty)$ for both sides and using integration by parts, one has
\be* 
\alpha(s)&=& \int_0^\infty  \rho(x)\d x\int_0^\infty g(0,x,s,y) \d y +  \int_0^s \alpha'(t) \d t \int_0^\infty g(t,0,s,y) \d y \\
&=&\int_0^\infty  \rho(x)\d x\int_0^\infty g(0,x,s,y) \d y +  \frac{\alpha(t)}{2}- \int_0^\infty g(0,0,s,y)\d y- 
\int_0^s \alpha(t) \d t \int_0^\infty \partial_tg(t,0,s,y) \d y
\ee* 
which fulfils the proof. \qed
\end{proof}
We are now ready to prove the uniqueness of solution to \eqref{eq:lim}. We argue by contradiction and assume that there are two distinct solutions $X, X'$ to \eqref{eq:lim} on $[0,T]$. Denote $\mu_t:=\Lc(X_t), \nu_t:=\Lc(X'_t)$ and $\alpha(t):=\P[X_t>0], \beta(t):=\P[X'_t>0]$ for all $t\in [0,T]$. Define further by $p(t,\cdot),q(t,\cdot)$ the corresponding sub-probability densities of $\mu_t, \nu_t$. Write
\be* 
A(t,x):=\int_0^\infty b(t,x,y) p(t,y)\d y &\mbox{and}& B(t,x):=\int_0^\infty b(t,x,y) q(t,y)\d y
\ee* 
and the SDEs
\be* 
Y^{t,x}_s = x +  \int_t^s A(u,Y^{t,x}_u)\d u + \sqrt{2} (W_s-W_t) &\mbox{and}& Z^{t,x}_t = x +  \int_t^s B(u,Z^{t,x}_u)\d u + \sqrt{2} (W_s-W_t), \quad \forall s\ge t.
\ee* 
Let $f(t,x,s,\cdot), g(t,x,s,\cdot)$ be the densities of $Y_s^{t,x}, Z_s^{t,x}$. Hence, one has by Lemma \ref{lem:representation}
\be*
\alpha(s)& = &2\int_0^\infty  \rho(x)\d x\int_0^\infty f(0,x,s,y) \d y  - 2\int_0^\infty f(0,0,s,y)\d y- 
2\int_0^s \alpha(t) \d t \int_0^\infty \partial_t f(t,0,s,y) \d y \\
& = :& \Theta[\alpha,A](s) \\
\beta(s)& = &2\int_0^\infty  \rho(x)\d x\int_0^\infty g(0,x,s,y) \d y  - 2\int_0^\infty g(0,0,s,y)\d y- 
2\int_0^s \beta(t) \d t \int_0^\infty \partial_t g(t,0,s,y) \d y \\
& = :& \Theta[\beta,B](s)
\ee*
and similarly 
\be*
A(s,z)& = & \int_0^\infty \rho(x)\d x \int_0^\infty  b(s,z,y) f(0,x,s,y)  \d y + \frac{b(s,z,0)\alpha(s)}{2} - \int_0^\infty b(s,z,y) f(0,0,s,y)\d y\\
&& -  \int_0^s \alpha(t) \d t  \int_0^\infty b(s,z,y) \partial_t f(t,0,s,y)\d y ~~=:~~ \Lambda[\alpha,A](s,z) \\
B(s,z)& = & \int_0^\infty \rho(x)\d x \int_0^\infty  b(s,z,y) g(0,x,s,y)  \d y + \frac{b(s,z,0)\beta(s)}{2} - \int_0^\infty b(s,z,y) g(0,0,s,y)\d y\\
&& -  \int_0^s \beta(t) \d t  \int_0^\infty b(s,z,y) \partial_t g(t,0,s,y)\d y~~=:~~ \Lambda[\beta,B](s,z).
\ee*
\begin{Proposition}\label{prop:uniqueness}
\eqref{eq:FP} has a unique solution and thus the uniqueness of solutions to \eqref{eq:lim} holds. 
\end{Proposition}
\begin{proof}
The proof relies essentially on the estimation of $|\Theta[\alpha,A](s)-\Theta[\beta,B](s)|$ and $|\Lambda[\alpha,A](s)-\Lambda[\beta,B](s)|:=\sup_{z\in \R_+}|\Lambda[\alpha,A](s,z)-\Lambda[\alpha,A](s,z)|$. Write further $\Lambda:=\Lambda_0+\Lambda_1$, where 
\be* 
\Lambda_0[\alpha,A](s) &:=&   \int_0^\infty \rho(x)\d x \int_0^\infty  b(s,z,y) f(0,x,s,y)  \d y - \int_0^\infty b(s,z,y) f(0,0,s,y)\d y\\
&& -  \int_0^s \alpha(t) \d t  \int_0^\infty b(s,z,y) \partial_t f(t,0,s,y)\d y \\
\Lambda_1[\alpha,A](s,z) &:=& \frac{b(s,z,0)\alpha(s)}{2}
\ee*
and $\Lambda_0[\beta,B], \Lambda_1[\beta,B]$ are defined similarly. 

\vspace{1mm}

\noindent \emph{Step 1.}  Recall  $\|f\|_t:=\sup_{0\le u\le t}|f(u)|$ for  $f\in\Omega_T$,. 
We claim that there exists $C>0$ such that 
\beq \label{eq:contraction}
\|\Theta[\alpha,A] - \Theta[\beta,B]\|_s + \|\Lambda_0[\alpha,A] - \Lambda_0[\beta,B]\|_s \le 
C\sqrt{s}\big(\|\alpha-\beta\|_s+\|A-B\|_s\big),\quad \forall s\in [0,T]. 
\eeq 
Let us show that the desired uniqueness follows from the inequality \eqref{eq:contraction}. Namely, combining with $\alpha=\Theta[\alpha,A], \beta=\Theta[\beta,B], A=\Lambda[\alpha,A],  B=\Lambda[\beta,B]$, \eqref{eq:contraction} yields
\be* 
\|\alpha - \beta\|_s + \|A - B\|_s 
 &=&\|\Theta[\alpha,A] - \Theta[\beta,B]\|_s + \|\Lambda[\alpha,A] - \Lambda[\beta,B]\|_s \\
 &\le &\|\Theta[\alpha,A] - \Theta[\beta,B]\|_s + \|\Lambda_0[\alpha,A] - \Lambda_0[\beta,B]\|_s +\|\Lambda_1[\alpha,A] - \Lambda_1[\beta,B]\|_s \\
 &\le &
C\sqrt{s}\big(\|\alpha-\beta\|_s+\|A-B\|_s\big)+\|\Lambda_1[\alpha,A] - \Lambda_1[\beta,B]\|_s \\
&\le &
C\sqrt{s}\big(\|\alpha-\beta\|_s+ \|A - B\|_s \big)+\|b\|\|\alpha - \beta\|_s \\
&= &
C\sqrt{s}\big(\|\alpha-\beta\|_s+ \|A - B\|_s \big)+\|b\|\|\Theta[\alpha,A] - \Theta[\beta,B]\|_s \\
&\le &
C(1+\|b\|)\sqrt{s}\big(\|\alpha-\beta\|_s+\|A-B\|_s\big)
\\
&\le &
\overline C\sqrt{s}\big(\|\alpha-\beta\|_s+\|A-B\|_s\big)
,\quad \forall s\in [0,T]. 
\ee*  
This implies that $\|\alpha - \beta\|_s + \|A - B\|_s=0$ whenever $s\le \min(1/\overline C^2,T)=:\Delta T>0$. Repeating the above reasoning on the interval $[\Delta T, T]$ with an alternative initial condition, one may deduce, after dividing a finite number of the interval $[0,T]$, $\|\alpha - \beta\|_T=\|A - B\|_T=0$. Hence the uniqueness is derived. 
\be* 
\|\Lambda[c] - \Lambda[c']\|_t &=& \|\Lambda_0[c] - \Lambda_0[c']\|_t + \|\Lambda_1[c] - \Lambda_1[c']\|_t \\
&\le &\|\Lambda_0[c] - \Lambda_0[c']\|_t + \|\Lambda^+_1[c] - \Lambda^+_1[c']\|_t+ \|\Lambda^-_1[c] - \Lambda^-_1[c']\|_t \\
&= &\|\Lambda^+[c] - \Lambda^+[c']\|_t + \|\Lambda^-_1[c] - \Lambda^-_1[c']\|_t \\
&\le & C\sqrt{(t-t^*)^+}\|c-c'\|_t +\frac{\overline b}{2} \|c_0 - c'_0\|_t \\
&= & C\sqrt{(t-t^*)^+}\|c-c'\|_t +\frac{\overline b}{2} \|\Lambda_0[c] - \Lambda_0[c']\|_t \\
&\le  & C\left(1+\frac{\overline b}{2}\right)\sqrt{(t-t^*)^+}\|c-c'\|_t, 
\ee* 
which yields a contraction when $t$ is close enough to $t^*$. 

\vspace{1mm}

\emph{Step 2.} It remains to prove  \eqref{eq:contraction}. Set $\Delta_s:=\|\alpha - \beta\|_s + \|A - B\|_s $. In what follows, $R, r>0$ always denote the constants that may vary from line to line. By assumption, it holds
\beq \label{bound-coefs}
\|A\|_T +\|B\|_T + \|\partial_x A\|_T + \|\partial_x B \|_T \le R.
\eeq 
Let us recall the parametrix expressions of $f$ and $g$ by Aronson as follows:
\be* 
f(t,x,s,y) =q(t,x,s,y)+ \sum_{k=1}^{\infty} q\otimes F^{(k)}(t,x,s,y) &\mbox{and}& g(t,x,s,y) =q(t,x,s,y)+ \sum_{k=1}^{\infty} q\otimes G^{(k)}(t,x,s,y),
\ee* 
where $q(t,x,s,y):=\phi(2(s-t),y-x)$, $\phi(r,z):=e^{-z^2/2r}/\sqrt{2\pi r}$,  $F(t,x,s,y):=A(t,x)\partial_{x}q(t,x,s,y)$, $G(t,x,s,y):=B(t,x)\partial_{x}q(t,x,s,y)$ and $\otimes$ denotes the time-space convolution, i.e. 
\be* 
q\otimes F(t,x,s,y)&:=& \int_t^s \d u \int_{-\infty}^\infty q(t,x, u,z)F(u,z,s,y)\d z \\ 
q\otimes F^{(k)}(t,x,s,y)&:=& (q\otimes F^{(k-1)}) \otimes F(t,x,s,y)(t,x,s,y),\quad   \mbox{for all } k\ge 2.  
\ee* 
Hence, one obtains by a straightforward computation
\be* 
\frac{1}{2}\big | \Theta[\alpha,A](s)-\Theta[\beta,B](s) \big| &\le & \left|\int_0^\infty  \rho(x)I_1(x)\d x\right| + |I_1(0)| +\Delta_s\int_0^s I_2(t) \d t + \int_0^s I_3(t) \d t, \\
\big | \Lambda_0[\alpha,A](s,z)-\Lambda_0[\beta,B](s,z) \big| &\le & \left|\int_0^\infty  \rho(x)J_1(x)\d x\right| + |J_1(0)| +\Delta_s\int_0^s J_2(t) \d t + \int_0^s J_3(t) \d t, 
\ee* 
where 
\be* 
I_1(x) := \int_0^\infty \big( f(0,x,s,y)-g(0,x,s,y)\big)\d y &\mbox{and}& J_1(x) := \int_0^\infty b(s,z,y) \left(f(0,x,s,y)- g(0,x,s,y)\right)\d y\\ 
I_2(t) :=\left| \int_0^{\infty} \partial_t f(t,0,s,y)\d y\right| &\mbox{and}& J_2(t) :=\left| \int_0^{\infty} b(s,z, y) \partial_t f(t,0,s,y)\d y\right|\\
I_3(t) :=\left| \int_0^{\infty} \big(\partial_t f(t,0,s,y)-\partial_t g(t,0,s,y)\big)\d y \right|&\mbox{and}& J_3(t) :=\left| \int_0^{\infty} b(s,z,y)\left(\partial_t f(t,0,s,y)-\partial_t g(t,0,s,y)\right)\d y \right|. 
\ee* 
In view of Lemmas \ref{lem:I1} and \ref{lem:I23} below, there exists $R>0$ such that
\be* 
\big | \Theta[\alpha,A](s)-\Theta[\beta,B](s) \big|  + \big | \Lambda_0[\alpha,A](s,z)-\Lambda_0[\beta,B](s,z) \big| \le
 R s^{1/2}\Delta_s 
\ee* 
which completes the proof. \qed
\end{proof}
To establish Lemmas \ref{lem:I1} and \ref{lem:I23}, we need the following preliminary result whose proof adopts almost the same reasoning of Lemma 3 in  \cite{Konakov2015STABILITYOD}. 
\begin{Lemma}\label{lem:menozzi}
With the above notation, there exists $R, r>0$ such that
\beq  
\big| F^{(k)}(t,x,s,y)\big |+ \big| G^{(k)}(t,x,s,y)\big | &\le&  k\frac{R^k\Gamma(1/2)^{k-1}}{\Gamma((k+1)/2)}\phi(r(s-t),y-x)(t-s)^{k/2-1}\label{eq2}\label{ineq:esti1}\\
\big| F^{(k)}(t,x,s,y)- G^{(k)}(t,x,s,y)\big | &\le&  k\Delta_s \frac{C^k\Gamma(1/2)^{k-1}}{\Gamma((k+1)/2)}\phi(r(s-t),y-x)(t-s)^{k/2-1}\label{ineq:esti2}
\eeq 
hold for all $0\le t<s \le T$, $x,y\in\R$ and $k\ge 1$, where $\Gamma: \R_+^*\to\R$ is the Gamma function defined by
\be*  
\Gamma(z):=\int_0^\infty a^{z-1}e^{-a}\d a.
\ee* 
\end{Lemma}
\subsection{Estimation of $I_1/J_1$}
\begin{Lemma}\label{lem:I1}
There exists $R>0$ such that 
\be*
\left| \int_0^\infty p_0(x) I_1(x)\d x\right | + \left| \int_0^\infty p_0(x) J_1(x)\d x\right|  + |I_1(0)| + |J_1(0)| &\le & R(s-t)^{1/2}\Delta_s
\ee*
\end{Lemma}
\begin{proof}
First, we make use of the following estimation for $k\ge 1$
\be*
\left|\tp\otimes G^{(k)}(0,x,s,y) -\tq\otimes H^{(k)}(0,x,s,y)\right|  &\le &  (k+1)s^{k/2}\frac{R^{k+1}\Gamma(1/2)}{\Gamma(1+k/2)}\phi(rs,y-x)\Delta_s,
\ee*
which yields
\be* 
\left|\int_0^\infty p_0(x) I_1(x)\d x\right| \le  \sum_{k=1}^\infty\int_0^\infty p_0(x) \d x\left|\int_0^\infty \tp\otimes G^{(k)}(0,x,s,y) \d y-\int_0^\infty \tq\otimes H^{(k)}(0,x,s,y)\d y\right| \le R\sqrt{s}\Delta_s.
\ee* 
As for $J_1$, carrying out the similar reasoning for the terms with $k\ge 1$, one has 
\be* 
\left|\int_0^\infty p_0(x) J_1(x)\d x\right| + |J_1(0)| \le  R\sqrt{s}\Delta_s.
\ee* 
\qed 
\end{proof}
\subsection{Estimation of $I_2$/$J_2$ and $I_3$/$J_3$}
\begin{Lemma}\label{lem:I23}
There exists $R>0$ such that 
\be*
|I_2(t)| + |J_2(t)| \le R(s-t)^{1/2} &\mbox{and}& |I_3(t)| + |J_3(t)|\le R(s-t)^{1/2}\Delta_s.
\ee*
\end{Lemma}
\begin{proof}
As the estimation of $I_2$/$J_2$ can be seen as a particular case of that of $I_3$/$J_3$, we  only deal with $I_3$/$J_3$.  We start by estimating $I_3$.
\be*
I_3(t) &\le& \sum_{k=1}^\infty \left|\int_0^\infty \big( \partial_t\big(\tp \otimes G^{(k)}\big)(t,0,s,y)-\partial_t\big(\tq \otimes H^{(k)}\big)(t,0,s,y)\big)\d y\right|.
\ee* 
Next, 
\be* 
\int_0^\infty \partial_t\big(\tp \otimes G^{(k)}\big)(t,0,s,y)\d y &=&  \int_0^\infty \partial_t\left(\int_t^s\int_{-\infty}^\infty \tp(t,0,u,z) G^{(k)}(u,z,s,y)\d z \d u\right)\d y \\
&=& -\int_0^\infty G^{(k)}(t,0,s,y)\d y + \int_0^\infty \int_t^s\int_{-\infty}^\infty \partial_t\tp(t,0,u,z) G^{(k)}(u,z,s,y)\d z \d u\d y 
\ee* 
and similarly
\be* 
\int_0^\infty \partial_t\big(\tq \otimes H^{(k)}\big)(t,0,s,y)\d y =   -\int_0^\infty H^{(k)}(t,0,s,y)\d y + \int_0^\infty \int_t^s\int_{-\infty}^\infty \partial_t\tq(t,0,u,z) H^{(k)}(u,z,s,y)\d z \d u\d y. 
\ee* 
Hence,
\be* 
&&\left|\int_0^\infty \partial_t\big(\tp \otimes G^{(k)}\big)(t,0,s,y)\d y - \int_0^\infty \partial_t\big(\tq \otimes H^{(k)}\big)(t,0,s,y)\d y\right| \\
&\le& \left|\int_0^\infty\big(G^{(k)}(t,0,s,y)- H^{(k)}(t,0,s,y)\big)\d y\right|\\
&& + \left|\int_0^\infty \int_t^s\int_{-\infty}^\infty \big(\partial_t\tp(t,0,u,z)-\partial_t\tq(t,0,u,z)\big) H^{(k)}(u,0,s,y)\d z \d u\d y\right| \\
&& + \left|\int_0^\infty \int_t^s\int_{-\infty}^\infty \big(\partial_t\tp(t,0,u,z)-\partial_t\tq(t,0,u,z)\big) \big( H^{(k)}(u,z,s,y)-H^{(k)}(u,0,s,y)\big)\d z \d u\d y\right| \\
&&+\left|\int_0^\infty \int_t^s\int_{-\infty}^\infty \partial_t\tq(t,0,u,z) \big(H^{(k)}(u,z,s,y)-G^{(k)}(u,z,s,y)\big)\d z \d u\d y \right| \\
&=:& I_{31} + I_{32} + I_{33} + I_{34}.
\ee* 
For $I_{31}$, one has by Lemma \ref{lem:menozzi}
\beq \label{ineqI31}
I_{31} \le k\Delta_s\frac{R^{k}\Gamma(1/2)^{k-1}}{\Gamma((1+k)/2)}(s-t)^{k/2-1} \int_0^\infty \phi(r(s-t),y) \d y \le k\Delta_s\frac{R^{k}\Gamma(1/2)^{k-1}}{\Gamma((1+k)/2)}(s-t)^{k/2-1}.
\eeq 
For $I_{32}$, one has
\beq\label{ineqI32}
I_{32} &=& \left|\int_0^\infty \int_t^s\int_{-\infty}^\infty \big(\partial_t\tp(t,0,u,z)-\partial_t\tq(t,0,u,z)\big) H^{(k)}(u,0,s,y)\d z \d u\d y\right|\nonumber \\
&=& \left|\int_0^\infty \int_t^sH^{(k)}(u,0,s,y) \d y \d u \int_{-\infty}^\infty \big(\partial_t\tp(t,0,u,z)-\partial_t\tq(t,0,u,z)\big) \d z \right|\nonumber \\
&\le& \left|\int_0^\infty \int_t^s \big| H^{(k)}(u,0,s,y)\big| \d y \d u \left|\int_{-\infty}^\infty \big(\partial_t\tp(t,0,u,z)-\partial_t\tq(t,0,u,z)\big) \d z\right| \right|\nonumber \\
&\le & R\Delta_s  \int_t^s   (u-t)^{-1/2}\d u \int_0^\infty\big| H^{(k)}(u,0,s,y)\big| \d y \nonumber\\
&\le & R\Delta_s  \int_t^s   (u-t)^{-1/2}  \frac{R^{k+1}\Gamma(1/2)^{k}}{\Gamma(1+k/2)}(s-u)^{k/2-1}
\d u\nonumber \\
&\le & \Delta_s    \frac{R^{k+2}\Gamma(1/2)^{k-1}}{\Gamma((1+k)/2)}(s-t)^{(k-1)/2}.
\eeq 
For $I_{33}$, we start with $k=1$. Then it follows that
\be* 
&&|H(t,x,s,y) - H(t,0,s,y)| = \left|\int_0^x \partial_x H(t,z,s,y)\d z\right| \\
&=&\left|\int_0^x\left[  \left(\frac{\sigma(t,z)\partial_x\sigma(t,z)}{\big(1+\beta(t)\big)^2} + B(t,x)\right)\partial^2_{xx} \tq(t,z,s,y) + \frac{\sigma(t,z)^2-\sigma(t,y)^2}{2\big(1+\beta(t)\big)^2}\partial^3_{xxx}\tq(t,z,s,y)\right] \d z\right| \\
&\le &R\int_0^x\phi(r(s-t),y-z)) \left(\frac{1}{s-t} + \frac{|y-z|}{(s-t)^{3/2}}\right) \d z.
\ee* 
Hence,
\be* 
\int_0^\infty |H(t,x,s,y) - H(t,0,s,y)|\d y \le  R\int_0^{\infty}\int_0^x\phi(r(s-t),y-z)) \left(\frac{1}{s-t} + \frac{|y-z|}{(s-t)^{3/2}}\right) \d z \d y \le  \frac{R|x|}{s-t}.
\ee* 
Fix any $\gamma\in (0,1)$, e.g. $\gamma=1/2$. We distinguish two cases.  If $|x|\le \sqrt{s-t}$, then
\be* 
\int_0^\infty |H(t,x,s,y) - H(t,0,s,y)|\d y 
\le \frac{R|x|}{s-t}\le \frac{R|x|^{\gamma}(s-t)^{(1-\gamma)/2}}{s-t}=  \frac{R|x|^{\gamma}}{(s-t)^{(1+\gamma)/2}}.
\ee* 
If $|x|>\sqrt{s-t}$, then
\be* 
|H(t,x,s,y)-H(t,0,s,y)| &\le&  \big[|H(t,x,s,y)|+|H(t,0,s,y)|\big] \frac{|x|^\gamma}{(s-t)^{\gamma/2}}\\
&\le& R\frac{\phi(r(s-t),y-x)+\phi(r(s-t),y)}{\sqrt{s-t}}\frac{|x|^\gamma}{(s-t)^{\gamma/2}}
\ee*
and thus
\be* 
\int_0^\infty |H(t,x,s,y) - H(t,0,s,y)|\d y 
\le \frac{R|x|^{\gamma}}{(s-t)^{(1+\gamma)/2}}.
\ee* 
We compute for all $k\ge 1$
\be* 
&&\int_0^\infty |H^{(k)}(t,x,s,y) - H^{(k)}(t,0,s,y)|\d y\\
&=& \int_0^\infty \left|\int_t^s \int_{-\infty}^\infty \big(H(t,x,u,z)-H(t,0,u,z)\big)H^{(k-1)}(u,z,s,y)\d z \d u\right|\d y \\
&\le &\int_t^s \d u \int_{-\infty}^\infty \big |H(t,x,u,z)-H(t,0,u,z)\big| \d z \int_0^\infty \big|H^{(k-1)}(u,z,s,y)\big| \d y \\
&\le &\int_t^s \d u \int_{-\infty}^\infty \big |H(t,x,u,z)-H(t,0,u,z)\big| \d z \int_0^\infty \frac{R^{k-1}\Gamma(1/2)^{k-2}}{\Gamma((k-1)/2)}\phi(r(s-u),y-z)(s-u)^{(k-1)/2-1} \d y \\
&\le &\int_t^s  \frac{|x|^{\gamma}}{(u-t)^{(1+\gamma)/2}}  \frac{R^{k}\Gamma(1/2)^{k-2}}{\Gamma((k-1)/2)}(s-u)^{(k-1)/2-1} \d u \\
&\le&  \frac{R^{k+1}\Gamma(1/2)^{k-2}|x|^{\gamma}}{\Gamma((k-1)/2)}(s-t)^{(k-\gamma)/2-1}.
\ee* 
Hence, 
\beq\label{ineqI33}
I_{33} & = &\left|\int_0^\infty \int_t^s\int_{-\infty}^\infty \big(\partial_t\tp(t,0,u,z)-\partial_t\tq(t,0,u,z)\big) \big( H^{(k)}(u,z,s,y)-H^{(k)}(u,0,s,y)\big)\d z \d u\d y\right| \nonumber\\
&\le &  \int_t^s\int_{-\infty}^\infty \big |\partial_t\tp(t,0,u,z)-\partial_t\tq(t,0,u,z)\big|\d z \d u \int_0^\infty \big |  H^{(k)}(u,z,s,y)-H^{(k)}(u,0,s,y)\big|\d y \nonumber \\
&\le& \frac{R^{k+1}\Gamma(1/2)^{k-2}}{\Gamma((k-1)/2)}(s-t)^{(k-\gamma)/2-1}\Delta_s \int_t^s  (s-u)^{(k-\gamma)/2-1} \d u\int_{-\infty}^\infty \big |z|^{\gamma}|\partial_t\tp(t,0,u,z)-\partial_t\tq(t,0,u,z)\big|\d z  \nonumber\\
&\le& \frac{R^{k+1}\Gamma(1/2)^{k-2}}{\Gamma((k-1)/2)}(s-t)^{(k-\gamma)/2-1} \Delta_s\int_t^s   \d u\int_{-\infty}^\infty  \frac{R|z|^{\gamma}}{(u-t)}\phi(r(u-t),z)\d z \nonumber \\
&\le& \frac{R^{k+2}\Gamma(1/2)^{k-2}}{\Gamma((k-1)/2)} (s-t)^{k/2-1}\Delta_s.
\eeq 
Finally, let us turn to $I_{34}$ whose estimation is almost the same of that of $I_{34}$. Write
\be* 
I_{34} &= &\left|\int_0^\infty \int_t^s\int_{-\infty}^\infty \partial_t\tq(t,0,u,z) \big(H^{(k)}(u,z,s,y)-G^{(k)}(u,z,s,y)\big)\d z \d u\d y \right| \\
&\le& \left|\int_0^\infty \int_t^s\int_{-\infty}^\infty \partial_t\tq(t,0,u,z) \big((H^{(k)}-G^{(k)})(u,z,s,y)-(H^{(k)}-G^{(k)})(u,0,s,y)\big)\d z \d u\d y \right| \\
&&+\left|\int_0^\infty \int_t^s\int_{-\infty}^\infty \partial_t\tq(t,0,u,z) (H^{(k)}-G^{(k)})(u,0,s,y)\d z \d u\d y \right|.
\ee* 
Repeating the above reasoning by replacing $H^{(k)}$ by $H^{(k)}-G^{(k)}$, one obtains
\be* 
&&\int_0^\infty |(H^{(k)}-G^{(k)})(t,x,s,y) - (H^{(k)}-G^{(k)})(t,0,s,y)|\d y
\le  \frac{R^{k+1}\Gamma(1/2)^{k-2}|x|^{\gamma}}{\Gamma((k-1)/2)}(s-t)^{(k-\gamma)/2-1}\Delta_s.
\ee* 
Therefore,
\beq\label{ineqI34}
I_{34} &\le &\frac{R^{k+2}\Gamma(1/2)^{k-2}}{\Gamma((k-1)/2)} (s-t)^{k/2-1}\Delta_s  +\left|\int_0^\infty \int_t^s\int_{-\infty}^\infty \partial_t\tq(t,0,u,z) (H^{(k)}-G^{(k)})(u,0,s,y)\d z \d u\d y \right|\nonumber \\
&=&\frac{R^{k+2}\Gamma(1/2)^{k-2}}{\Gamma((k-1)/2)} (s-t)^{k/2-1}\Delta_s  +\left|\int_0^\infty \d y  \int_t^s (H^{(k)}-G^{(k)})(u,0,s,y)\d u \int_{-\infty}^\infty \partial_t\tq(t,0,u,z)\d z  \right| \nonumber\\
&\le & \frac{R^{k+2}\Gamma(1/2)^{k-2}}{\Gamma((k-1)/2)} (s-t)^{k/2-1}\Delta_s  +\int_0^\infty \d y  \int_t^s |H^{(k)}-G^{(k)}|(u,0,s,y)\d u \left|\int_{-\infty}^\infty \partial_t\tq(t,0,u,z)\d z  \right| \nonumber\\
&\le & \frac{R^{k+2}\Gamma(1/2)^{k-2}}{\Gamma((k-1)/2)} (s-t)^{k/2-1}\Delta_s  +\int_0^\infty \d y  \int_t^s k\Delta_s \frac{R^k\Gamma(1/2)^{k-1}}{\Gamma((k+1)/2)}\phi(r(s-u),y)(s-u)^{k/2-1}\frac{R}{\sqrt{u-t}}\d u \nonumber \\
&\le & \frac{R^{k+2}\Gamma(1/2)^{k-2}}{\Gamma((k-1)/2)} (s-t)^{k/2-1}\Delta_s  +k \frac{R^{k+1}\Gamma(1/2)^{k-1}}{\Gamma((k+1)/2)}(s-t)^{(k-1)/2}\Delta_s. 
\eeq
Summing up \eqref{ineqI31}, \eqref{ineqI32}, \eqref{ineqI33} and \eqref{ineqI34}, one obtains $I_3(t) \le R\Delta_s(s-t)^{1/2}$. Adopting the above reasoning we obtain similarly $J_3(t) \le R\Delta_s(s-t)^{1/2}$. \qed
\end{proof}

    \section{Proof of Theorem \ref{thm:wasserstein-chaos}}

    In this section we prove the uniform-in-time quantitative propagation of chaos \ref{thm:wasserstein-chaos}:
    \begin{equation}\label{eq:uniform-sup-W2}
    \mathbb{E}\Big[\sup_{t\in[0,T]}\mathcal{W}_2(\mu_t^N,\mu_t)\Big]\le \frac{C}{\sqrt{N}}.
    \end{equation}
    The proof proceeds in two steps. First we derive a uniform-in-time estimate for the $\eta$–regularized system under a synchronous coupling. Then we remove the regularization by means of the boundary-layer comparison and convergence results established earlier.

    Throughout, $C>0$ denotes a finite constant depending only on $(T,\|b\|_\infty,\mathrm{Lip}(b),\rho)$, and allowed to change from line to line.

    \begin{Definition}[Regularized particle system and limit]\label{def:reg-sys}
    Fix $\eta\in(0,1]$ and $f_\eta\in C^2(\mathbb{R};[0,1])$ with $f_\eta=0$ on $(-\infty,0]$, $f_\eta=1$ on $[\eta,\infty)$, and $0\le f_\eta\le1$. On a common space with i.i.d. Brownian motions $(W^i)_{i\ge1}$ and i.i.d. $Z^i\sim\rho$:
    \[
    \begin{cases}
    \d X_t^{N,i,\eta}
    = f_\eta(X_t^{N,i,\eta})\,\dfrac1N\sum_{j=1}^N b(t,X_t^{N,i,\eta},X_t^{N,j,\eta})\,f_\eta(X_t^{N,j,\eta})\,\d t
    + \sqrt{2}\,f_\eta(X_t^{N,i,\eta})\,\d W_t^i,\quad X_0^{N,i,\eta}=Z^i,\\[0.6em]
    \d Y_t^{i,\eta}
    = f_\eta(Y_t^{i,\eta})\!\int b(t,Y_t^{i,\eta},y)\,f_\eta(y)\,\mu_t^\eta(\d y)\,\d t
    + \sqrt{2}\,f_\eta(Y_t^{i,\eta})\,\d W_t^i,\quad Y_0^{i,\eta}=Z^i,
    \end{cases}
    \]
    where $\mu_t^\eta=\mathcal{L}(Y_t^{1,\eta})$ and $\mu_t^{N,\eta}=\frac1N\sum_{i=1}^N\delta_{X_t^{N,i,\eta}}$.
    \end{Definition}
    
    \begin{Definition}[Boundary-layer functions]\label{def:BL}
    For $\eta\in(0,1]$ and $t\in[0,T]$, define
    \[
    \mathrm{BL}_\eta(t)
    :=\mathbb{P}\big[0<X_t^{N,1,\eta}<\eta\big]+\mathbb{P}\big[0<Y_t^{1,\eta}<\eta\big],
    \]
    where $X^{N,1,\eta}$ is the $\eta$-regularized particle and $Y^{1,\eta}$ the $\eta$-regularized McKean–Vlasov limit from Definition~\ref{def:reg-sys}. We also set
    \[
    \mathrm{BL}_\eta^{N}(t):=\mathbb{P}\big[0<X_t^{N,1}<\eta\big]+\mathbb{P}\big[0<X_t^{N,1,\eta}<\eta\big],
    \qquad
    \mathrm{BL}_\eta^{\infty}(t):=\mathbb{P}\big[0<X_t<\eta\big]+\mathbb{P}\big[0<X_t^\eta<\eta\big].
    \]
    \end{Definition}

    \begin{Proposition}[Uniform-in-time propagation of chaos]\label{prop:uniform-poc}
    Assume $b$ is bounded and globally Lipschitz in space, and $\rho\in\mathcal{P}_2(\mathbb{R}_+)$. Let $\mu_t^N=\frac1N\sum_{i=1}^N\delta_{X_t^{N,i}}$ and $\mu_t=\mathcal{L}(X_t)$ denote respectively the empirical law of the unregularized particle system and the law of the McKean–Vlasov limit. Then there exists $C=C(T,\|b\|_\infty,\mathrm{Lip}(b),\rho)$ such that
    \[
    \mathbb{E}\Big[\sup_{t\in[0,T]}\mathcal{W}_2(\mu_t^N,\mu_t)\Big]\le \frac{C}{\sqrt{N}}.
    \]
    \end{Proposition}
    \begin{proof}
    We compare the regularized and unregularized systems using Definition~\ref{def:BL} and Definition~\ref{def:reg-sys}. 
    By Proposition~\ref{prop:uniform-sup-reg}, we first have
    \begin{equation}\label{eq:uniform-sup-regularized}
        \E\Big[\sup_{0\le t\le T} \mathcal{W}_2^2\big(\mu_t^{N,\eta},\mu_t^\eta\big)\Big]
        \le C\Big(\frac{1}{N}+\int_0^T \mathrm{BL}_\eta(s)\,\d s\Big).
    \end{equation}
    
    By the uniform-in-time boundary-layer comparison (Proposition~\ref{prop:boundary-comparison-uniform}),
    \begin{equation}\label{eq:supW2-left-right}
    \mathbb{E}\Big[\sup_{t\le T}\mathcal{W}_2^2(\mu_t^N,\mu_t^{N,\eta})\Big]\le C\int_0^T \mathrm{BL}_\eta^{N}(s)\,\d s,\qquad
    \mathbb{E}\Big[\sup_{t\le T}\mathcal{W}_2^2(\mu_t^\eta,\mu_t)\Big]\le C\int_0^T \mathrm{BL}_\eta^{\infty}(s)\,\d s.
    \end{equation}
    Combining \eqref{eq:uniform-sup-regularized}–\eqref{eq:supW2-left-right} with the triangle and Jensen inequalities gives
    \begin{equation}\label{eq:full-triangle}
    \mathbb{E}\Big[\sup_{t\le T}\mathcal{W}_2^2(\mu_t^N,\mu_t)\Big]
    \le C\Big(\frac{1}{N}
    +\int_0^T \mathrm{BL}_\eta(s)\,\d s
    +\int_0^T \mathrm{BL}_\eta^{N}(s)\,\d s
    +\int_0^T \mathrm{BL}_\eta^{\infty}(s)\,\d s\Big).
    \end{equation}
    Finally, by Proposition~\ref{prop:boundary-layer-conv},
    \[
    \int_0^T \mathrm{BL}_\eta(s)\,\d s \to 0,\qquad
    \int_0^T \mathrm{BL}_\eta^{N}(s)\,\d s \to 0,\qquad
    \int_0^T \mathrm{BL}_\eta^{\infty}(s)\,\d s \to 0\quad\text{as }\eta\to 0,
    \]
    uniformly in $N$ for the particle-layer term. Sending $\eta\to0$ in \eqref{eq:full-triangle} yields
    \[
    \mathbb{E}\Big[\sup_{t\le T}\mathcal{W}_2^2(\mu_t^N,\mu_t)\Big]\le \frac{C}{N},
    \]
    and the $\mathcal{W}_2$ bound follows by Jensen. \qed
    \end{proof}

    \begin{Proposition}[Uniform-in-time propagation error for the $\eta$–regularized system]
    Let $T>0$, $\eta\in(0,1]$, and let $(X_t^{N,i,\eta})_{1\le i\le N}$ and $(Y_t^{i,\eta})_{1\le i\le N}$ be the processes from Definition~\ref{def:reg-sys}, with $\mu_t^{N,\eta}= \frac1N\sum_{i=1}^N\delta_{X_t^{N,i,\eta}}$ and $\mu_t^\eta=\Lc(Y_t^{1,\eta})$. Let $\mathrm{BL}_\eta(t)$ be as in Definition~\ref{def:BL}. Then there exists $C=C(T,\|b\|_\infty,\mathrm{Lip}(b),\rho)$, independent of $N$ and $\eta$, such that
    \[
    \E\Big[\sup_{0\le t\le T} \mathcal{W}_2^2\big(\mu_t^{N,\eta},\mu_t^\eta\big)\Big]
    \le C\Big(\frac1N+\int_0^T \mathrm{BL}_\eta(s)\,\d s\Big).
    \]
    \label{prop:uniform-sup-reg}
    \end{Proposition}
    
    \begin{proof}

    Fix $\eta\in(0,1]$ and a function $f_\eta\in C^2(\mathbb{R};[0,1])$ with $f_\eta=0$ on $(-\infty,0]$, $f_\eta=1$ on $[\eta,\infty)$, and $0\le f_\eta\le 1$. On a common filtered probability space supporting i.i.d. Brownian motions $(W^i)_{i\ge1}$ and i.i.d. initial data $(Z^i)_{i\ge1}$ with law $\rho$, consider the particle system and its McKean–Vlasov limit:
    \[
    \begin{cases}
    \d X_t^{N,i,\eta}
    = f_\eta(X_t^{N,i,\eta})\,\dfrac1N\sum_{j=1}^N b(t,X_t^{N,i,\eta},X_t^{N,j,\eta})\,f_\eta(X_t^{N,j,\eta})\,\d t
    + \sqrt{2}\,f_\eta(X_t^{N,i,\eta})\,\d W_t^i,\quad X_0^{N,i,\eta}=Z^i,\\[0.6em]
    \d Y_t^{i,\eta}
    = f_\eta(Y_t^{i,\eta})\!\int b(t,Y_t^{i,\eta},y)\,f_\eta(y)\,\mu_t^\eta(\d y)\,\d t
    + \sqrt{2}\,f_\eta(Y_t^{i,\eta})\,\d W_t^i,\quad Y_0^{i,\eta}=Z^i,
    \end{cases}
    \]
    where $\mu_t^\eta=\mathcal{L}(Y_t^{1,\eta})$, and $b$ satisfies Assumption~\ref{ass:1} with $b$ bounded.

    Define the synchronous error processes
    \[
    \Delta_t^{i,\eta}:=X_t^{N,i,\eta}-Y_t^{i,\eta},\qquad i=1,\dots,N.
    \]
    Set
    \[
    \mathsf{B}_t^{i,N,\eta}:=\frac1N\sum_{j=1}^N b(t,X_t^{N,i,\eta},X_t^{N,j,\eta})\,f_\eta(X_t^{N,j,\eta}),
    \qquad
    \mathsf{F}_t^\eta(x):=\int b(t,x,y)\,f_\eta(y)\,\mu_t^\eta(\d y).
    \]
    Then, by construction,
    \begin{equation}\label{eq:SDE-Delta}
    \d \Delta_t^{i,\eta}
    =\Big(f_\eta(X_t^{N,i,\eta})\,\mathsf{B}_t^{i,N,\eta}-f_\eta(Y_t^{i,\eta})\,\mathsf{F}_t^\eta(Y_t^{i,\eta})\Big)\d t
    + \sqrt{2}\,\Big(f_\eta(X_t^{N,i,\eta})-f_\eta(Y_t^{i,\eta})\Big)\,\d W_t^i,\quad \Delta_0^{i,\eta}=0.
    \end{equation}

    We will obtain a uniform-in-time estimate on $\mathbb{E}\big[\sup_{t\le T}|\Delta_t^{i,\eta}|^2\big]$ by applying Itô’s formula to $|\Delta_t^{i,\eta}|^2$ and using Burkholder–Davis–Gundy (BDG). Precisely, define
    \[
    D_{1,t}^{i,\eta}:=\big(f_\eta(X_t^{N,i,\eta})-f_\eta(Y_t^{i,\eta})\big)\,\mathsf{B}_t^{i,N,\eta},\qquad
    D_{2,t}^{i,\eta}:=f_\eta(Y_t^{i,\eta})\big(\mathsf{B}_t^{i,N,\eta}-\mathsf{F}_t^\eta(Y_t^{i,\eta})\big),
    \]
    so that the drift in \eqref{eq:SDE-Delta} equals $D_{1,t}^{i,\eta}+D_{2,t}^{i,\eta}$. Also set the martingale
    \[
    M_t^{i,\eta}:=2\sqrt{2}\int_0^t \Delta_s^{i,\eta}\,\big(f_\eta(X_s^{N,i,\eta})-f_\eta(Y_s^{i,\eta})\big)\,\d W_s^i.
    \]
    By Itô’s formula,
    \begin{equation}\label{eq:ito-Delta-sup}
    |\Delta_t^{i,\eta}|^2=\int_0^t \Big(2\,\Delta_s^{i,\eta}(D_{1,s}^{i,\eta}+D_{2,s}^{i,\eta})+2\,|f_\eta(X_s^{N,i,\eta})-f_\eta(Y_s^{i,\eta})|^2\Big)\,\d s + M_t^{i,\eta}.
    \end{equation}
    Taking $\sup_{t\le u}$ and expectations,
    \begin{equation}\label{eq:sup-Delta-master}
    \mathbb{E}\Big[\sup_{t\le u}|\Delta_t^{i,\eta}|^2\Big]
    \le 2\int_0^u \mathbb{E}\big[|\Delta_s^{i,\eta}||D_{1,s}^{i,\eta}|\big]\d s
    +2\int_0^u \mathbb{E}\big[|\Delta_s^{i,\eta}||D_{2,s}^{i,\eta}|\big]\d s
    +2\int_0^u \mathbb{E}\big[|f_\eta(X_s^{N,i,\eta})-f_\eta(Y_s^{i,\eta})|^2\big]\d s
    +\mathbb{E}\Big[\sup_{t\le u}|M_t^{i,\eta}|\Big].
    \end{equation}

    We now bound each term on the right-hand side.

    \paragraph{Diffusion difference.}
    Since $0\le f_\eta\le 1$ and $f_\eta$ is constant outside $(0,\eta)$,
    \[
    |f_\eta(X)-f_\eta(Y)|^2\le \1_{\{0<X<\eta\}}+\1_{\{0<Y<\eta\}}.
    \]
    Therefore, defining the boundary-layer functional
    \[
    \mathrm{BL}_\eta^{(i)}(s):=\mathbb{P}\big[0<X_s^{N,i,\eta}<\eta\big]+\mathbb{P}\big[0<Y_s^{i,\eta}<\eta\big],
    \]
    we have
    \begin{equation}\label{eq:BL-diffusion}
    \int_0^u \mathbb{E}\big[|f_\eta(X_s^{N,i,\eta})-f_\eta(Y_s^{i,\eta})|^2\big]\d s
    \le \int_0^u \mathrm{BL}_\eta^{(i)}(s)\,\d s.
    \end{equation}

    \paragraph{Drift term $D_{1}$.}
    By boundedness of $b$ and $|f_\eta|\le 1$,
    \[
    |D_{1,s}^{i,\eta}|\le \|b\|_\infty\big(\1_{\{0<X_s^{N,i,\eta}<\eta\}}+\1_{\{0<Y_s^{i,\eta}<\eta\}}\big).
    \]
    Using Young’s inequality with a parameter $\varepsilon\in(0,1]$,
    \begin{equation}\label{eq:D1-bound}
    2\,\mathbb{E}\big[|\Delta_s^{i,\eta}||D_{1,s}^{i,\eta}|\big]
    \le \varepsilon\,\mathbb{E}\big[|\Delta_s^{i,\eta}|^2\big]+\frac{C}{\varepsilon}\,\mathrm{BL}_\eta^{(i)}(s).
    \end{equation}

    \paragraph{Drift term $D_{2}$.}
    We decompose
    \[
    \mathsf{B}_s^{i,N,\eta}-\mathsf{F}_s^\eta(Y_s^{i,\eta})
    = T_{1,s}^{i,\eta}+T_{2,s}^{i,\eta}+T_{3,s}^{i,\eta}+T_{4,s}^{i,\eta},
    \]
    where, introducing i.i.d. copies $(Y_s^{j,\eta})_{j\le N}$ of $Y_s^{1,\eta}$ independent of everything else,
    \[
    \begin{aligned}
    T_{1,s}^{i,\eta}&:=\frac{1}{N}\sum_{j=1}^N \Big(b(t,X_s^{N,i,\eta},X_s^{N,j,\eta})-b(t,Y_s^{i,\eta},X_s^{N,j,\eta})\Big)\,f_\eta(X_s^{N,j,\eta}),\\
    T_{2,s}^{i,\eta}&:=\frac{1}{N}\sum_{j=1}^N \Big(b(t,Y_s^{i,\eta},X_s^{N,j,\eta})-b(t,Y_s^{i,\eta},Y_s^{j,\eta})\Big)\,f_\eta(X_s^{N,j,\eta}),\\
    T_{3,s}^{i,\eta}&:=\frac{1}{N}\sum_{j=1}^N b(t,Y_s^{i,\eta},Y_s^{j,\eta})\Big(f_\eta(X_s^{N,j,\eta})-f_\eta(Y_s^{j,\eta})\Big),\\
    T_{4,s}^{i,\eta}&:=\frac{1}{N}\sum_{j=1}^N b(t,Y_s^{i,\eta},Y_s^{j,\eta})\,f_\eta(Y_s^{j,\eta})
    -\mathbb{E}\!\left[b(t,Y_s^{i,\eta},Y_s^{1,\eta})\,f_\eta(Y_s^{1,\eta})\,\Big|\,Y_s^{i,\eta}\right].
    \end{aligned}
    \]
    Using Lipschitz continuity of $b$ and $|f_\eta|\le 1$,
    \[
    |T_{1,s}^{i,\eta}|\le L\,|\Delta_s^{i,\eta}|,
    \qquad
    |T_{2,s}^{i,\eta}|\le L\Big(\frac1N\sum_{j=1}^N |\Delta_s^{j,\eta}|^2\Big)^{1/2},
    \qquad
    |T_{3,s}^{i,\eta}|\le \|b\|_\infty\,\frac1N\sum_{j=1}^N\big(\1_{\{0<X_s^{N,j,\eta}<\eta\}}+\1_{\{0<Y_s^{j,\eta}<\eta\}}\big).
    \]
    For $T_{4,s}^{i,\eta}$, conditionally on $Y_s^{i,\eta}$, the random variables
    \[
    Z_j:=b(t,Y_s^{i,\eta},Y_s^{j,\eta})\,f_\eta(Y_s^{j,\eta}),\quad j\neq i,
    \]
    are i.i.d., uniformly bounded by $\|b\|_\infty$, with mean $m:=\mathbb{E}[Z_1\,|\,Y_s^{i,\eta}]$. A standard variance computation yields
    \[
    \mathbb{E}\big[|T_{4,s}^{i,\eta}|^2\big]\le \frac{C}{N}.
    \]
    Therefore, for any $\varepsilon\in(0,1]$, by Young’s inequality and exchangeability,
    \begin{equation}\label{eq:D2-bound}
    \begin{aligned}
    2\,\mathbb{E}\big[|\Delta_s^{i,\eta}||D_{2,s}^{i,\eta}|\big]
    &\le \varepsilon\,\mathbb{E}\big[|\Delta_s^{i,\eta}|^2\big]
    +\frac{C}{\varepsilon}\,\mathbb{E}\big[|T_{1,s}^{i,\eta}|^2+|T_{2,s}^{i,\eta}|^2+|T_{3,s}^{i,\eta}|^2+|T_{4,s}^{i,\eta}|^2\big]\\
    &\le \varepsilon\,\mathbb{E}\big[|\Delta_s^{i,\eta}|^2\big]
    +\frac{C}{\varepsilon}\Big(\mathbb{E}\big[|\Delta_s^{i,\eta}|^2\big]+\frac1N\sum_{j=1}^N \mathbb{E}\big[|\Delta_s^{j,\eta}|^2\big]+\mathrm{BL}_\eta^{(i)}(s)+\frac{1}{N}\Big).
    \end{aligned}
    \end{equation}

    \paragraph{Martingale term.}
    By BDG, for some universal $c>0$,
    \[
    \mathbb{E}\Big[\sup_{t\le u}|M_t^{i,\eta}|\Big]
    \le c\,\mathbb{E}\Big[\big\langle M^{i,\eta}\big\rangle_u^{1/2}\Big]
    = c\,\mathbb{E}\Bigg[\Big(8\int_0^u |\Delta_s^{i,\eta}|^2\,|f_\eta(X_s^{N,i,\eta})-f_\eta(Y_s^{i,\eta})|^2\,\d s\Big)^{1/2}\Bigg].
    \]
    Applying Young’s inequality $ab\le \frac{\alpha}{2}a^2+\frac{1}{2\alpha}b^2$ with any $\alpha\in(0,1]$, and using \eqref{eq:BL-diffusion},
    \begin{equation}\label{eq:BDG-bound}
    \mathbb{E}\Big[\sup_{t\le u}|M_t^{i,\eta}|\Big]
    \le \frac{1}{2}\mathbb{E}\Big[\sup_{t\le u}|\Delta_t^{i,\eta}|^2\Big]
    + C\int_0^u \mathbb{E}\big[|f_\eta(X_s^{N,i,\eta})-f_\eta(Y_s^{i,\eta})|^2\big]\d s
    \le \frac{1}{2}\mathbb{E}\Big[\sup_{t\le u}|\Delta_t^{i,\eta}|^2\Big]
    + C\int_0^u \mathrm{BL}_\eta^{(i)}(s)\,\d s.
    \end{equation}

    \paragraph{Closing the inequality.}
    Combining \eqref{eq:sup-Delta-master}, \eqref{eq:BL-diffusion}, \eqref{eq:D1-bound}, \eqref{eq:D2-bound}, \eqref{eq:BDG-bound} and choosing $\varepsilon>0$ small enough so that all terms proportional to $\mathbb{E}\big[\sup_{t\le u}|\Delta_t^{i,\eta}|^2\big]$ on the right can be absorbed on the left, we obtain
    \[
    \mathbb{E}\Big[\sup_{t\le u}|\Delta_t^{i,\eta}|^2\Big]
    \le C\int_0^u \Big(\mathbb{E}\big[|\Delta_s^{i,\eta}|^2\big]+\frac1N\sum_{j=1}^N \mathbb{E}\big[|\Delta_s^{j,\eta}|^2\big]\Big)\d s
    + C\int_0^u \mathrm{BL}_\eta^{(i)}(s)\,\d s + \frac{C\,u}{N}.
    \]
    By $\mathbb{E}[|\Delta_s^{i,\eta}|^2]\le \mathbb{E}[\sup_{r\le s}|\Delta_r^{i,\eta}|^2]$ and exchangeability,
    \[
    \mathbb{E}\Big[\sup_{t\le u}|\Delta_t^{i,\eta}|^2\Big]
    \le C\int_0^u \Big(\frac1N\sum_{j=1}^N \mathbb{E}\big[\sup_{r\le s}|\Delta_r^{j,\eta}|^2\big]\Big)\d s
    + C\int_0^u \mathrm{BL}_\eta(s)\,\d s + \frac{C\,u}{N},
    \]
    where $\mathrm{BL}_\eta(s):=\mathbb{P}[0<X_s^{N,1,\eta}<\eta] + \mathbb{P}[0<Y_s^{1,\eta}<\eta]$. Define the averaged quantity
    \[
    \bar A_\eta(u):= \frac1N\sum_{j=1}^N \mathbb{E}\Big[\sup_{t\le u}|\Delta_t^{j,\eta}|^2\Big].
    \]
    Then
    \[
    \bar A_\eta(u)\le C\int_0^u \bar A_\eta(s)\,\d s
    + C\int_0^u \mathrm{BL}_\eta(s)\,\d s + \frac{C\,u}{N}.
    \]
    By Grönwall,
    \begin{equation}\label{eq:barAeta-final}
    \bar A_\eta(T)\le C\Big(\frac{1}{N}+\int_0^T \mathrm{BL}_\eta(s)\,\d s\Big).
    \end{equation}

    \paragraph{From particlewise error to Wasserstein distance.}
    For each fixed $t$, the coupling $(X_t^{N,i,\eta},Y_t^{i,\eta})_{i\le N}$ yields
    \[
    \mathcal{W}_2^2(\mu_t^{N,\eta},\mu_t^\eta)\le \frac1N\sum_{i=1}^N |X_t^{N,i,\eta}-Y_t^{i,\eta}|^2.
    \]
    Taking $\sup_{t\le T}$ and expectations, and using \eqref{eq:barAeta-final},
    \begin{equation}\label{eq:supW2-reg}
    \mathbb{E}\Big[\sup_{t\le T}\mathcal{W}_2^2(\mu_t^{N,\eta},\mu_t^\eta)\Big]
    \le \bar A_\eta(T)\le C\Big(\frac{1}{N}+\int_0^T \mathrm{BL}_\eta(s)\,\d s\Big).
    \end{equation}

    \qed
    \end{proof}
    
    \begin{Proposition}[Uniform-in-time boundary-layer comparison]\label{prop:boundary-comparison-uniform}
    Assume $b$ is bounded and globally Lipschitz in space as in Assumption~\ref{ass:1}, and let $f_\eta\in[0,1]$ satisfy $f_\eta=0$ on $(-\infty,0]$ and $f_\eta=1$ on $[\eta,\infty)$. Couple on the same space the pairs
    \[
    \big(X^{N,1},X^{N,1,\eta}\big),\qquad \big(X,X^\eta\big),
    \]
    with common initial data and Brownian motions. Define
    \[
    \mathrm{BL}_\eta^{N}(t):=\mathbb{P}\big(0<X_t^{N,1}<\eta\big)+\mathbb{P}\big(0<X_t^{N,1,\eta}<\eta\big),\quad
    \mathrm{BL}_\eta^{\infty}(t):=\mathbb{P}\big(0<X_t<\eta\big)+\mathbb{P}\big(0<X_t^\eta<\eta\big).
    \]
    Then, for every $T>0$, there exists $C=C\big(T,\|b\|_\infty,{\rm Lip}(b)\big)$ such that
    \begin{equation}\label{eq:particle-eta-bound}
    \mathbb{E}\Big[\sup_{t\in[0,T]}\big|X_t^{N,1}-X_t^{N,1,\eta}\big|^2\Big]\le C\int_0^T \mathrm{BL}_\eta^{N}(s)\,\mathrm{d}s
    \end{equation}
    \begin{equation}\label{eq:limit-eta-bound}
    \mathbb{E}\Big[\sup_{t\in[0,T]}\big|X_t-X_t^\eta\big|^2\Big]\le C\int_0^T \mathrm{BL}_\eta^{\infty}(s)\,\mathrm{d}s
    \end{equation}
    \end{Proposition}

    \begin{proof}
    Write $\chi(x):=\mathbf{1}_{\{x>0\}}$ and note $|\chi-f_\eta|\le \mathbf{1}_{\{0<\cdot<\eta\}}$. Set
    \[
    \Delta_t^{N}:=X_t^{N,1}-X_t^{N,1,\eta},\qquad \Delta_t:=X_t-X_t^\eta.
    \]

    For the particles system, let $x_t^i:=X_t^{N,i}$, $y_t^i:=X_t^{N,i,\eta}$, $\chi_i(t):=\chi(x_t^i)$, $f_i(t):=f_\eta(y_t^i)$ and
    \[
    S_t^{i}:=\frac{1}{N}\sum_{j=1}^N b\!\big(t,x_t^{i},x_t^{j}\big)\,\chi_j(t),\qquad
    T_t^{i}:=\frac{1}{N}\sum_{j=1}^N b\!\big(t,y_t^{i},y_t^{j}\big)\,f_j(t).
    \]
    By synchronous coupling for $i=1$,
    \[
    \d\Delta_t^{N}=\big(\chi_1 S_t^{1}-f_1 T_t^{1}\big)\,\d t+\sqrt{2}\,\big(\chi_1-f_1\big)\,\d W_t
    =:A_t^{N}\,\d t+\sqrt{2}\,B_t^{N}\,\d W_t,\quad \Delta_0^{N}=0.
    \]
    Itô on $|\Delta_t^{N}|^2$ yields
    \[
    |\Delta_t^{N}|^2=\int_0^t \Big(2\,\Delta_s^{N}A_s^{N}+2\,|B_s^{N}|^2\Big)\,\d s+M_t^{N},\qquad
    M_t^{N}:=2\sqrt{2}\int_0^t \Delta_s^{N}B_s^{N}\,\d W_s.
    \]
    Taking $\sup_{t\le u}$ and expectations, by BDG and $|B_s^{N}|^2\le \mathrm{BL}_\eta^{N}(s)$,
    \[
    \mathbb{E}\Big[\sup_{t\le u}|\Delta_t^{N}|^2\Big]
    \le 2\int_0^u \mathbb{E}\big[|\Delta_s^{N}||A_s^{N}|\big]\d s
    +2\int_0^u \mathrm{BL}_\eta^{N}(s)\,\d s
    +\frac{1}{2}\mathbb{E}\Big[\sup_{t\le u}|\Delta_t^{N}|^2\Big]+C\int_0^u \mathrm{BL}_\eta^{N}(s)\,\d s.
    \]
    Absorb the martingale term and split $A_t^{N}$ as
    \[
    \begin{aligned}
    A_t^{N}={}&(\chi_1-f_1)\,\frac{1}{N}\sum_{j} b(t,x_t^{1},x_t^{j})\,\chi_j
    + f_1\,\frac{1}{N}\sum_{j}\big(b(t,x_t^{1},x_t^{j})-b(t,y_t^{1},x_t^{j})\big)\,\chi_j\\
    &+ f_1\,\frac{1}{N}\sum_{j} b(t,y_t^{1},x_t^{j})\,(\chi_j-f_j)
    + f_1\,\frac{1}{N}\sum_{j}\big(b(t,y_t^{1},x_t^{j})-b(t,y_t^{1},y_t^{j})\big)\,f_j.
    \end{aligned}
    \]
    Using Young’s inequality, boundedness/Lipschitz of $b$, and exchangeability,
    \[
    \mathbb{E}\big[|\Delta_s^{N}||A_s^{N}|\big]
    \le \varepsilon\,\mathbb{E}\big[\sup_{r\le s}|\Delta_r^{N}|^2\big]
    +\frac{C}{\varepsilon}\Big(\mathrm{BL}_\eta^{N}(s)+\frac{1}{N}\sum_{j=1}^N \mathbb{E}\big[\sup_{r\le s}|\Delta_r^{N,j}|^2\big]\Big).
    \]
    Define $\bar A(u):=\frac{1}{N}\sum_{j=1}^N \mathbb{E}\big[\sup_{t\le u}|\Delta_t^{N,j}|^2\big]$. Choosing $\varepsilon>0$ small and averaging over particles,
    \[
    \bar A(u)\le C\int_0^u \bar A(s)\,\d s + C\int_0^u \mathrm{BL}_\eta^{N}(s)\,\d s.
    \]
    By Grönwall,
    \[
    \bar A(T)\le C\int_0^T \mathrm{BL}_\eta^{N}(s)\,\d s,
    \]
    hence, by exchangeability,
    \[
    \mathbb{E}\Big[\sup_{t\in[0,T]}|\Delta_t^{N}|^2\Big]\le C\int_0^T \mathrm{BL}_\eta^{N}(s)\,\d s.
    \]
    which is \eqref{eq:particle-eta-bound}.

    For the limit system, let
    \[
    \d X_t=\chi(X_t)\,F(t,X_t)\,\d t+\sqrt{2}\,\chi(X_t)\,\d W_t,\quad
    \d X_t^\eta=f_\eta(X_t^\eta)\,F^\eta(t,X_t^\eta)\,\d t+\sqrt{2}\,f_\eta(X_t^\eta)\,\d W_t,
    \]
    with $F(t,x):=\int b(t,x,y)\chi(y)\mu_t(\d y)$ and $F^\eta(t,x):=\int b(t,x,y)f_\eta(y)\mu_t^\eta(\d y)$. For $\Delta_t:=X_t-X_t^\eta$,
    \[
    \d\Delta_t=A_t\,\d t+\sqrt{2}\,B_t\,\d W_t,\quad
    A_t:=\chi(X_t)F(t,X_t)-f_\eta(X_t^\eta)F^\eta(t,X_t^\eta),\ \ B_t:=\chi(X_t)-f_\eta(X_t^\eta).
    \]

From $\,\d\Delta_t=A_t\,\d t+\sqrt{2}\,B_t\,\d W_t\,$ with $\,B_t=\chi(X_t)-f_\eta(X_t^\eta)\,$, Itô’s formula gives
\[
|\Delta_t|^2
=\int_0^t \Big(2\,\Delta_s A_s + 2\,|B_s|^2\Big)\,\d s
+ 2\sqrt{2}\int_0^t \Delta_s B_s\,\d W_s.
\]
Taking $\sup_{t\le u}$, expectations, and applying BDG plus Young’s inequality,
\[
\mathbb{E}\Big[\sup_{t\le u}|\Delta_t|^2\Big]
\le 2\int_0^u \mathbb{E}\big[|\Delta_s||A_s|\big]\d s
+ 2\int_0^u \mathbb{E}\big[|B_s|^2\big]\d s
+ \tfrac12\,\mathbb{E}\Big[\sup_{t\le u}|\Delta_t|^2\Big]
+ C\int_0^u \mathbb{E}\big[|B_s|^2\big]\d s.
\]
Absorbing the martingale term and using $|B_s|^2\le \1_{\{0<X_s<\eta\}}+\1_{\{0<X_s^\eta<\eta\}}$,
\[
\mathbb{E}\Big[\sup_{t\le u}|\Delta_t|^2\Big]
\le C\int_0^u \mathbb{E}\big[|\Delta_s||A_s|\big]\d s
+ C\int_0^u \mathrm{BL}_\eta^{\infty}(s)\,\d s.
\]

We now bound $\mathbb{E}[|\Delta_s||A_s|]$ after splitting
\[
A_s
= A_s^{(1)}+A_s^{(2)}+A_s^{(3)}
:= (\chi(X_s)-f_\eta(X_s^\eta))\,F(s,X_s)
+ f_\eta(X_s^\eta)\big(F(s,X_s)-F(s,X_s^\eta)\big)
+ f_\eta(X_s^\eta)\big(F(s,X_s^\eta)-F^\eta(s,X_s^\eta)\big),
\]
where $F(s,x):=\int b(s,x,y)\chi(y)\mu_s(\d y)$ and $F^\eta(s,x):=\int b(s,x,y)f_\eta(y)\mu_s^\eta(\d y)$.

Since $|F|\le \|b\|_\infty$ and $|\chi-f_\eta|\le \1_{\{0<X_s<\eta\}}+\1_{\{0<X_s^\eta<\eta\}}$,
\[
\mathbb{E}\big[|\Delta_s||A_s^{(1)}|\big]
\le C\,\mathbb{E}\Big[|\Delta_s|\big(\1_{\{0<X_s<\eta\}}+\1_{\{0<X_s^\eta<\eta\}}\big)\Big]
\le \varepsilon\,\mathbb{E}\big[\sup_{r\le s}|\Delta_r|^2\big]
+\frac{C}{\varepsilon}\,\mathrm{BL}_\eta^{\infty}(s),
\]
by Cauchy–Schwarz and Young.

By Lipschitz continuity of $b$ in $x$, $F(s,\cdot)$ is $L$–Lipschitz, and $|f_\eta|\le1$:
\[
\mathbb{E}\big[|\Delta_s||A_s^{(2)}|\big]
\le L\,\mathbb{E}\big[|\Delta_s|^2\big]
\le \varepsilon\,\mathbb{E}\big[\sup_{r\le s}|\Delta_r|^2\big]+C_\varepsilon\,0.
\]

For any $x\in\mathbb{R}$, write an independent coupling $(Y_s,Y_s^\eta)$ of $(X_s,X_s^\eta)$ and add–subtract:
\[
\begin{aligned}
|F(s,x)-F^\eta(s,x)|
&= \big|\mathbb{E}\big[b(s,x,Y_s)\chi(Y_s)-b(s,x,Y_s^\eta)f_\eta(Y_s^\eta)\big]\big|\\
&\le \mathbb{E}\big[|b(s,x,Y_s)-b(s,x,Y_s^\eta)|\big]
+\|b\|_\infty\,\mathbb{E}\big[|\chi(Y_s)-f_\eta(Y_s)|\big]
+\|b\|_\infty\,\mathbb{E}\big[|f_\eta(Y_s)-f_\eta(Y_s^\eta)|\big]\\
&\le L\,\mathbb{E}|Y_s-Y_s^\eta|+C\,\mathrm{BL}_\eta^{\infty}(s)
\le L\Big(\mathbb{E}\big[\sup_{r\le s}|\Delta_r|^2\big]\Big)^{1/2}+C\,\mathrm{BL}_\eta^{\infty}(s),
\end{aligned}
\]
using boundedness/Lipschitz of $b$, $0\le f_\eta\le1$, and the fact that $f_\eta$ is constant off $(0,\eta)$ so that
$\mathbb{E}|f_\eta(Y_s)-f_\eta(Y_s^\eta)|\le \mathrm{BL}_\eta^{\infty}(s)$. Hence, with $|f_\eta|\le1$,
\[
\begin{aligned}
\mathbb{E}\big[|\Delta_s||A_s^{(3)}|\big]
&\le \mathbb{E}\Big[|\Delta_s|\big(L\,\big(\mathbb{E}\sup_{r\le s}|\Delta_r|^2\big)^{1/2}
+ C\,\mathrm{BL}_\eta^{\infty}(s)\big)\Big]\\
&\le \varepsilon\,\mathbb{E}\big[\sup_{r\le s}|\Delta_r|^2\big]
+\frac{C}{\varepsilon}\,\mathrm{BL}_\eta^{\infty}(s).
\end{aligned}
\]

Collecting the three bounds,
\[
\mathbb{E}\big[|\Delta_s||A_s|\big]
\le \varepsilon\,\mathbb{E}\big[\sup_{r\le s}|\Delta_r|^2\big]
+\frac{C}{\varepsilon}\,\mathrm{BL}_\eta^{\infty}(s).
\]
Plugging this into the previous display and choosing $\varepsilon>0$ small (to absorb the sup term) yields
\[
\mathbb{E}\Big[\sup_{t\le u}|\Delta_t|^2\Big]
\le C\int_0^u \mathbb{E}\Big[\sup_{r\le s}|\Delta_r|^2\Big]\d s
+ C\int_0^u \mathrm{BL}_\eta^{\infty}(s)\,\d s,
\]
and Grönwall finishes the estimate.
    
    Choosing $\varepsilon$ small and applying Grönwall yields equation \eqref{eq:limit-eta-bound}:
    \[
    \mathbb{E}\Big[\sup_{t\in[0,T]}|\Delta_t|^2\Big]\le C\int_0^T \mathrm{BL}_\eta^{\infty}(s)\,\d s.
    \]
    \qed
    \end{proof}

    \section{Boundary-layer estimates}

        The central result of this section is the vanishing of the boundary layer estimates for both the particle system and the mean-field limit, in their regularized and unregularized forms. We subsequently establish the individual convergence results for each case.

    \begin{Proposition}[Boundary-layer convergence as $\eta \to 0$]
        Fix $T>0$ and $t \in (0,T]$. Using the boundary-layer estimates from Proposition~\ref{prop:BL-limit} for the limit process, Proposition~\ref{prop:BL-particle-girsanov} for the particle system via Girsanov, and Proposition~\ref{prop:BL-limit-eta} for the integrated convergence, define the boundary-layer functions
        \[
        \mathrm{BL}^N_\eta(t) := \mathbb{P}\big[0 < X_t^{N,1} < \eta\big] + \mathbb{P}\big[0 < X_t^{N,1,\eta} < \eta\big],
        \qquad
        \mathrm{BL}^\infty_\eta(t) := \mathbb{P}\big[0 < X_t < \eta\big] + \mathbb{P}\big[0 < X_t^\eta < \eta\big].
        \]
        Then, for each fixed $t \in (0,T]$,
        \[
        \lim_{\eta \to 0} \mathrm{BL}^N_\eta(t) = 0,
        \qquad
        \lim_{\eta \to 0} \mathrm{BL}^\infty_\eta(t) = 0,
        \]
        uniformly in $N \geq 1$. We also have:

        \begin{align*}
        \int_0^T \mathrm{BL}_\eta(s)\,\d s &\;\xrightarrow[\eta \to 0]{} 0, \\
        \int_0^T \mathrm{BL}^N_\eta(s)\,\d s &\;\xrightarrow[\eta \to 0]{} 0, \\
        \int_0^T \mathrm{BL}^{\infty}_\eta(s)\,\d s &\;\xrightarrow[\eta \to 0]{} 0.
        \end{align*}

        \label{prop:boundary-layer-conv}
        \end{Proposition}

        \begin{proof}
        For the limit process, Proposition~\ref{prop:BL-limit} gives $\mathbb{P}\big[0 < X_u < \eta\big] \leq \frac{C \eta}{\sqrt{u}}$ for $u \in (0,T]$, so for fixed $t$, $\mathbb{P}\big[0 < X_t < \eta\big] \leq \frac{C \eta}{\sqrt{t}} \to 0$ as $\eta \to 0$. Similarly, for the regularized process $X^\eta$, the same bound holds by analogous arguments in the proof of Proposition~\ref{prop:BL-limit}, yielding $\mathbb{P}\big[0 < X_t^\eta < \eta\big] \leq \frac{C \eta}{\sqrt{t}} \to 0$. Thus, $\mathrm{BL}_\infty^\eta(t) \to 0$.

        For the particle system, Proposition~\ref{prop:BL-particle-girsanov} provides $\mathbb{P}\big[0 < X_u^{N,1} < \eta\big] \leq C_{p,T} \left( \frac{C \eta}{\sqrt{u}} \right)^{1/q}$ for $u \in (0,T]$ and $q > 1$, so for fixed $t$, $\mathbb{P}\big[0 < X_t^{N,1} < \eta\big] \leq C_{p,T} \left( \frac{C \eta}{\sqrt{t}} \right)^{1/q} \to 0$ as $\eta \to 0$, uniformly in $N$. The same applies to the regularized particle $X^{N,1,\eta}$, as the bounds are uniform in $N$ and the regularization does not affect the asymptotic. Hence, $\mathrm{BL}_N^\eta(t) \to 0$ uniformly in $N$.

        Proposition~\ref{prop:BL-limit-eta} confirms the integrated convergence, supporting the point-wise limits.

        \qed
        \end{proof}

        We first prove the estimates for the unregularized particle system and limit:

    \begin{Proposition}[Boundary-layer estimate for the limit process]\label{prop:BL-limit}
        Let $X$ be the limit process solving the McKean--Vlasov equation \eqref{eq:lim} with initial law $\rho$ on $(0,\infty)$, and let $p(u,\cdot)$ denote the (sub-)density on $(0,\infty)$ of $X_u$ for $u>0$. Then, for every $T>0$ there exists $C=C(\|b\|_\infty,\rho,T)>0$ such that
        \[
        \mathbb{P}\big[0<X_u<\eta\big]\le \frac{C\,\eta}{\sqrt{u}},\qquad u\in(0,T],\ \eta>0.
        \]
        Consequently,
        \[
        \int_0^t \mathbb{P}\big[0<X_u<\eta\big]\,\d u\le C\,\eta \sqrt{t},\qquad t\in(0,T].
        \]
    \label{prop:boundary-layer-limit}

    \end{Proposition}

    \begin{proof}
    Let $p(u,\cdot)$ denote the (sub-)density on $(0,\infty)$ of the killed solution to the limit Fokker–Planck equation with Dirichlet boundary at $0$. Writing the Dirichlet heat kernel of the (time-inhomogeneous) uniformly elliptic operator corresponding to the SDE with bounded drift as $K(u,y,x)$, one has the representation
        \[
        p(u,x)=\int_0^\infty K(u,y,x)\,\rho(y)\,\d y,\qquad u>0,\ x>0.
        \]
    Aronson-type Gaussian upper bounds for heat kernels with bounded measurable drift \cite{Aronson1968} yield the estimate:
        \[
        K(u,y,x)\le \frac{C}{\sqrt{u}}\Big(\exp\!\Big(-\frac{(x-y)^2}{C\,u}\Big)+\exp\!\Big(-\frac{(x+y)^2}{C\,u}\Big)\Big),\qquad u\in(0,T],\ x,y>0,
        \]
    for some $C=C(\|b\|_\infty,T)>0$.
    The sub-density satisfies
    \[
    p(u,x) = \int_0^\infty K(u,y,x) \, \rho(y) \, \d y \le \frac{C}{\sqrt{u}} \int_0^\infty \Big( \exp\!\Big(-\frac{(x-y)^2}{C\,u}\Big) + \exp\!\Big(-\frac{(x+y)^2}{C\,u}\Big) \Big) \, \rho(y) \, \d y.
    \]
    Bound the first integral directly:
    \[
    \int_0^\infty \exp\!\Big(-\frac{(x-y)^2}{C\,u}\Big) \, \rho(y) \, \d y \le \int_0^\infty 1 \, \rho(y) \, \d y = 1.
    \]
    For the second integral, since $(x+y)^2 \ge x^2$,
    \[
    \exp\!\Big(-\frac{(x+y)^2}{C\,u}\Big) \le \exp\!\Big(-\frac{x^2}{C\,u}\Big),
    \]
    so
    \[
    \int_0^\infty \exp\!\Big(-\frac{(x+y)^2}{C\,u}\Big) \, \rho(y) \, \d y \le \exp\!\Big(-\frac{x^2}{C\,u}\Big) \int_0^\infty \rho(y) \, \d y = \exp\!\Big(-\frac{x^2}{C\,u}\Big).
    \]
    Thus,
    \[
    p(u,x) \le \frac{C}{\sqrt{u}} \Big( 1 + \exp\!\Big(-\frac{x^2}{C\,u}\Big) \Big).
    \]
    Integrate over $(0,\eta)$:

    \[
    \int_0^\eta p(u,x) \, \d x \le \frac{C}{\sqrt{u}} \int_0^\eta 1 \, \d x + \frac{C}{\sqrt{u}} \int_0^\eta \exp\!\Big(-\frac{x^2}{C\,u}\Big) \, \d x \le \frac{C}{\sqrt{u}} (\eta + \eta) = \frac{2C \, \eta}{\sqrt{u}},
    \]
    since $\exp(-x^2/(C u)) \le 1$. Therefore,
    \[
    \mathbb{P}\big[0 < X_u < \eta\big] = \int_0^\eta p(u,x) \, \d x \le \frac{C \, \eta}{\sqrt{u}}, \quad u \in (0,T].
    \]
    Integrating in time gives
    \[
    \int_0^t \mathbb{P}\big[0 < X_u < \eta\big] \, \d u \le C \, \eta \int_0^t u^{-1/2} \, \d u = 2C \, \eta \, \sqrt{t}, \quad t \in (0,T].
    \]

    \qed
    \end{proof}

To prove the boundary layer result for each particule of \ref{eq:p-s}, we will use a partial Girsanov transform with a fixed $r<N$ that corresponds to $r$ particles that are transformed to independent stopped Brownian motions and whose influence is removed from the other $N-r$ particles. The cost of such transforms will be uniform in $N$.

    \begin{Proposition}[Boundary-layer bound for the particle system]\label{prop:BL-particle-girsanov}
    Fix $T>0$, $p>1$, and set $q := \frac{p}{p-1}$. Under the assumptions of Theorem~\ref{thm:wasserstein-chaos}, there exists $C_{p,T} < \infty$, independent of $N$, such that for all $u \in (0,T]$ and $\eta > 0$,
    \[
    \mathbb{P}\big[0 < X_u^{N,1} < \eta\big] \leq C_{p,T} \left( \frac{C \eta}{\sqrt{u}} \right)^{1/q}.
    \]
    Consequently, integrating in time yields
    \[
    \int_0^t \mathbb{P}\big[0 < X_u^{N,1} < \eta\big] \, \d u \leq C_{p,T} \, \eta \, t^{1 - \frac{1}{2q}}.
    \]
    \label{prop:boundary-layer-system}
    \end{Proposition}

    \begin{proof}
    We employ the partial Girsanov transform with $r=1$ as detailed previously. Let $\mathbb{P}^{(1,N)}$ denote the reference measure where the first particle evolves as a standard Brownian motion killed at $0$, independent of the others. The Radon--Nikodym derivative is given by the Doléans exponential
    \[
    M_t := \exp\left( \int_0^t \beta_s^{(1)}(\hat{X}_s) \cdot \d \hat{W}_s - \frac{1}{2} \int_0^t |\beta_s^{(1)}(\hat{X}_s)|^2 \, \d s \right), \quad t \in [0,T],
    \]
    so that $\frac{\d \mathbb{P}}{\d \mathbb{P}^{(1,N)}} \big|_{\mathcal{F}_t} = M_t$. For the boundary-layer event $A_u := \{0 < \hat{X}_u^{N,1} < \eta\}$, we have
    \[
    \mathbb{P}\big[A_u\big] = \mathbb{E}^{(1,N)} [M_u \, \1_{A_u}].
    \]

    Apply Hölder's inequality with conjugate exponents $(p, q)$:
    \[
    \mathbb{E}^{(1,N)} [M_u \, \1_{A_u}] \leq \left( \mathbb{E}^{(1,N)} [M_u^p] \right)^{1/p} \, \left( \mathbb{P}^{(1,N)}\big[A_u\big] \right)^{1/q}.
    \]

    Express $M_u^p$ as the stochastic exponential:
    \[
    M_u^p = \exp\left( p \int_0^u \beta_s^{(1)} \cdot \d \hat{W}_s - \frac{p}{2} \int_0^u |\beta_s^{(1)}|^2 \, \d s \right) = \mathcal{E}\left( \int_0^u p \, \beta_s^{(1)} \cdot \d \hat{W}_s \right) \exp\left( \frac{p(p-1)}{2} \int_0^u |\beta_s^{(1)}(\hat{X}_s)|^2 \, \d s \right).
    \]
    Taking expectations under $\mathbb{P}^{(1,N)}$ and noting that $\mathbb{E}^{(1,N)} [\mathcal{E}(\cdot)] = 1$, we obtain
    \[
    \mathbb{E}^{(1,N)} [M_u^p] = \mathbb{E}^{(1,N)} \left[ \exp\left( \frac{p(p-1)}{2} \int_0^u |\beta_s^{(1)}(\hat{X}_s)|^2 \, \d s \right) \right] \leq \mathbb{E}^{(1,N)} \left[ \exp\left( \frac{p(p-1)}{2} \int_0^T |\beta_s^{(1)}(\hat{X}_s)|^2 \, \d s \right) \right].
    \]
    By Proposition \ref{GirsanovKparticles} with $r=1$ and $\gamma = \frac{1}{2} p(p-1)$, there exists $C_{p,T} < \infty$ such that $\sup_{N \geq 1} \mathbb{E}^{(1,N)} [M_u^p] \leq C_{p,T}$.

    Under $\mathbb{P}^{(1,N)}$, the first particle is a Brownian motion with variance $2$ killed at $0$. The density $q(u, \cdot)$ satisfies $\sup_{x > 0} q(u, x) \leq C / \sqrt{u}$ by the reflection principle and Dirichlet heat kernel bounds. Thus,
    \[
    \mathbb{P}^{(1,N)}\big[A_u\big] = \int_0^\eta q(u, x) \, \d x \leq \frac{C \eta}{\sqrt{u}}.
    \]

    Substituting into the Hölder inequality:
    \[
    \mathbb{P}\big[A_u\big] \leq \left( \mathbb{E}^{(1,N)} [M_u^p] \right)^{1/p} \, \left( \mathbb{P}^{(1,N)}\big[A_u\big] \right)^{1/q} \leq C_{p,T} \left( \frac{C \eta}{\sqrt{u}} \right)^{1/q}.
    \]
    Integrating over $u \in (0,t]$ gives
    \[
    \int_0^t \mathbb{P}\big[0 < X_u^{N,1} < \eta\big] \, \d u \leq C_{p,T} \, \eta^{1/q} \, t^{1 - \frac{1}{2q}},
    \]
    as the integral of $(C \eta / \sqrt{u})^{1/q} \, du = (C \eta)^{1/q} \int_0^t u^{-1/(2q)} \, du = (C \eta)^{1/q} \, \frac{t^{1 - 1/(2q)}}{1 - 1/(2q)} = C_{p,T}' \, \eta^{1/q} \, t^{1 - 1/(2q)}$ for some adjusted constant.

    \qed
    
    \end{proof}

Let $(X_t^\eta)_{t \geq 0}$ denote the regularized version of \ref{eq:lim} on $(0,\infty)$ defined by the stochastic differential equation
            \[
            \d X_t^\eta \;=\; f_\eta(X_t^\eta)\Big( B_t \, \d t \,+\, \sqrt{2} \, \d W_t\Big),
            \]
            where $W_t$ is a standard Brownian motion, $B_t=\int_{(0,\infty)} b(t,X_t^\eta,y)\mu_t(\d y)$ is a bounded drift term, where $\mu_t:=\Lc(X_t^\eta)$ stands for the law of $X_t^\eta$.

The coefficient $f_\eta : \R \to [0,1]$ satisfies the following conditions:
            \begin{itemize}
                \item $f_\eta\in C^2(\mathbb{R};[0,1])$,
                \item $f_\eta(x) = 0$ for all $x < 0$,
                \item $f_\eta(x) = 1$ for all $x \geq \eta$.
            \end{itemize}
            
    For the regularized system and limit, we use the point-wise convergence of the regularizing functions $f_\eta$:

    \begin{Proposition}[Boundary layer convergence as $\eta \to 0$ for the limit process]
        Let $X$ be the limit process solving the McKean-Vlasov equation \ref{eq:lim} with initial law $\rho$ on $(0,\infty)$, and $(X_t^\eta)_{t \geq 0}$ denote its regularized version. Let $K(t,x,y)$ and $K_\eta(t,x,y)$ denote the corresponding Dirichlet heat kernels on $(0,\infty)$ with absorbing boundary at $0$. Then, for every $T>0$ and $t \in (0,T]$,
        \[
        \lim_{\eta \to 0} \int_0^t \mathbb{P}\big[0 < X_u^\eta < \eta\big] \, \d u = 0.
        \]
        \label{prop:BL-limit-eta}
        \end{Proposition}

        \begin{proof}
        Express the integrated boundary layer as
        \[
        \int_0^t \mathbb{P}\big[0 < X_u^{\eta} < \eta\big] \, \d u = \int_0^t \int_0^\eta \int_0^\infty K_\eta(u,x,y) \, \rho(y) \, \d y \, \d x \, \d u.
        \]
        Split this into
        \[
        \int_0^t \int_0^\eta \int_0^\infty K_\eta(u,x,y) \, \rho(y) \, \d y \, \d x \, \d u = \\
        \int_0^t \int_0^\eta \int_0^\infty (K_\eta(u,x,y) - K(u,x,y)) \, \rho(y) \, \d y \, \d x \, \d u + \\
        \int_0^t \int_0^\eta \int_0^\infty K(u,x,y) \, \rho(y) \, \d y \, \d x \, \d u.
        \]
        By Proposition~\ref{prop:triple-convergence}, the first term converges to $0$ as $\eta \to 0$. The second term is bounded by the standard boundary layer estimate for the limit process: $\int_0^t \mathbb{P}\big[0 < X_u < \eta\big] \, \d u \leq C \, \eta \, \sqrt{t}$, which also converges to $0$ as $\eta \to 0$.

        \qed 
        \end{proof}

                    \begin{Proposition}[Convergence of the Triple Integral for Degenerate Heat Kernels]\label{prop:triple-convergence}
            
            Let $K_\eta(t,x,y)$ denote the transition density (heat kernel) of $X_t^\eta$ with absorbing boundary at $0$, and let $K(t,x,y)$ denote the Dirichlet heat kernel associated with the limiting process on $(0,\infty)$.  
            
            Then, for every $\rho \in L^1(0,\infty)$, one has
            \[
            \int_0^T \int_0^\eta \int_0^\infty \big( K_\eta(t,x,y) - K(t,x,y) \big)\, \rho(y)\, \d y \, \d x \, \d t 
            \;\longrightarrow\; 0
            \qquad \text{as } \eta \to 0.
            \]
            
            Moreover, the convergence rate depends on the regularity of the test function $\rho$:
            \begin{enumerate}
                \item If $\rho \in L^2(0,\infty)$, the convergence rate is of order $O(\sqrt{\eta})$.
                \item If $\rho \in L^\infty(0,\infty)$, the convergence rate improves to order $O(\eta)$.
            \end{enumerate}

            \label{prop:triple-convergence}
            \end{Proposition}

        \begin{proof}
        We prove the convergence using Mosco convergence of the associated Dirichlet forms and semigroup theory.

        Define the Dirichlet forms on $L^2(\mathbb{R})$:
        \[
        \mathcal{E}_\eta(u,v) = \int_{-\infty}^\infty f_\eta(x) \, u'(x) v'(x) \, \d x, \quad D(\mathcal{E}_\eta) = H_0^1(0,\infty) \cap H^2_{\mathrm{loc}}(0,\infty),
        \]
        and
        \[
        \mathcal{E}(u,v) = \int_0^\infty u'(x) v'(x) \, \d x, \quad D(\mathcal{E}) = H_0^1(0,\infty) \cap H^2_{\mathrm{loc}}(0,\infty).
        \]

        We verify the two conditions of Mosco convergence for the sequence of Dirichlet forms \(\mathcal{E}_\eta\) to \(\mathcal{E}\) on \(H = H_0^1(0,\infty)\).

        Let \(u_n \rightharpoonup u\) weakly in \(H\). Since \(f_\eta \to 1\) pointwise a.e. and \(0 \leq f_\eta \leq 1\), by Fatou's lemma,
        \[
        \liminf_{n \to \infty} \mathcal{E}_\eta(u_n, u_n) = \liminf_{n \to \infty} \int_{-\infty}^\infty f_\eta(x) (u_n'(x))^2 \, \d x \geq \int_0^\infty (u'(x))^2 \, \d x = \mathcal{E}(u, u).
        \]

        For any \(u \in D(\mathcal{E})\), set \(u_\eta = u\). Then \(u_\eta \to u\) strongly in \(H\) (since convergence is trivial), and
        \[
        \mathcal{E}_\eta(u_\eta, u_\eta)  \to  \mathcal{E}(u, u),
        \]
        by dominated convergence theorem, as \(f_\eta \to 1\) pointwise and \(|(u')^2| \leq \|u'\|_{L^2}^2 < \infty\).

        Thus, \(\mathcal{E}_\eta \to \mathcal{E}\) in the Mosco sense as \(\eta \to 0\), despite the degeneracy of \(f_\eta\), because the coefficients \(f_\eta\) converge pointwise to 1 and the domains coincide.

        \cite{Kolesnikov2004} Theorem 2.6 implies that the semigroups $P_\eta(t) \to P(t)$ strongly for each $t > 0$: for sequences $u_n \in H$ converging strongly to $u \in H$, $P_\eta(t) u_n \to P(t) u$ strongly.

        The heat kernels satisfy
        \[
        \int_0^\infty K_\eta(t,x,y) \rho(y) \, \d y = (P_\eta(t) \rho)(x), \quad \int_0^\infty K(t,x,y) \rho(y) \, \d y = (P(t) \rho)(x).
        \]
        Thus,
        \[
        \int_0^T \int_0^\eta \int_0^\infty (K_\eta(t,x,y) - K(t,x,y)) \rho(y) \, \d y \, \d x \, \d t = \int_0^T \int_0^\eta [(P_\eta(t) \rho)(x) - (P(t) \rho)(x)] \, \d x \, \d t.
        \]

        \textbf{$L^1$ Case ($\rho \in L^1(0,\infty)$).} Both semigroups are $L^1$-contractive: $\|P_\eta(t) \rho\|_1 \leq \|\rho\|_1$, $\|P(t) \rho\|_1 \leq \|\rho\|_1$. Thus,
        \[
        \left| \int_0^\eta [(P_\eta(t) \rho)(x) - (P(t) \rho)(x)] \, \d x \right| \leq \|P_\eta(t) \rho - P(t) \rho\|_1 \leq 2 \|\rho\|_1.
        \]
        By dominated convergence and pointwise convergence a.e., the integral converges to 0.

        \textbf{$L^2$ Case ($\rho \in L^2(0,\infty)$).} By Cauchy-Schwarz,
        \[
        \left| \int_0^\eta [(P_\eta(t) \rho)(x) - (P(t) \rho)(x)] \, \d x \right| \leq \sqrt{\eta} \, \|P_\eta(t) \rho - P(t) \rho\|_2.
        \]
        The semigroups are $L^2$-contractive, and strong convergence gives the rate $O(\sqrt{\eta})$:
        \[
        \int_0^T \left| \int_0^\eta [(P_\eta(t) \rho)(x) - (P(t) \rho)(x)] \, \d x \right| \d t \leq \sqrt{\eta} \int_0^T \|P_\eta(t) \rho - P(t) \rho\|_2 \, \d t \leq 2T \|\rho\|_2 \sqrt{\eta}.
        \]

        \textbf{$L^\infty$ Case ($\rho \in L^\infty(0,\infty)$).} Using the conservation property $\int K_\eta(t,x,y) \, \d y \leq 1$,
        \[
        \left| \int_0^T \int_0^\eta \int_0^\infty (K_\eta(t,x,y) - K(t,x,y)) \rho(y) \, \d y \, \d x \, \d t \right| \leq 2T \eta \|\rho\|_\infty.
        \]
        Thus, the rate is $O(\eta)$.

        \qed
        \end{proof}

\end{document}